\documentclass[12pt]{article}

\usepackage{e-jc}
\usepackage{amsmath}
\usepackage[utf8]{inputenc}
\usepackage[T1]{fontenc}
\usepackage{ae,aecompl}
\usepackage{enumerate}
\usepackage[active]{srcltx}     %
\usepackage{ifpdf}
\ifpdf
\fi
\usepackage{array}
\usepackage{graphicx}
\usepackage{rotating}
\usepackage{lineno}

\title%
{Some relational structures with polynomial growth and their
  associated algebras I:\\
{Quasi-polynomiality of the profile}}

\author{Maurice Pouzet\thanks{This work was done under the auspices
  of the Intas programme Universal algebra and Lattice theory, and
  supported by CMCU Franco-Tunisien "Outils math\'ematiques pour
  l'informatique".}\\
\small  Universit\'e  Claude-Bernard, Lyon1, ICJ,\\[-0.8ex]
\small $43$, Bd. du $11$ Novembre $1918$, Villeurbanne, F-69622\\
\small Department of Mathematics and Statistics, The University of Calgary,\\[-0.8ex]
\small Calgary, Alberta, Canada\\
\small \texttt{pouzet@univ-lyon1.fr}\\
\and
Nicolas M. Thi\'ery\\
\small Universit\'e Paris-Sud, Laboratoire de Math\'ematiques d'Orsay, Orsay, F-91405;\\[-0.8ex]
\small CNRS, Orsay, F-91405\\
\small \texttt{nthiery@users.sf.net}\\
\small \textit{Dedicated to Adriano Garsia, with warmth and admiration}
}

\specs{1}{20}{2013}

\date{%
  \small Mathematics Subject Classification: 05A15, 05E15, 03E20\\
  March 20, 2014%
}

\usepackage{color}
\usepackage{ifthen}
\newboolean{draft}
\setboolean{draft}{true}
\ifdraft
\newcommand{\TODO}[2][To do: ]{\textcolor{red}{\textbf{#1#2}}}
\else
\newcommand{\TODO}[2][]{}
\fi

\usepackage{amsfonts,amsthm,amssymb}
\usepackage{hyperref}

\newcommand{\suchthat}{{\ :\ }}

\newcommand{\isection}{{\downarrow}}
\newcommand{\sg}{{\mathfrak S}}

\def\shuff#1#2{\mathbin{
      \hbox{\vbox{
        \hbox{\vrule
              \hskip#2
              \vrule height#1 width 0pt
               }%
        \hrule}%
             \vbox{
        \hbox{\vrule
              \hskip#2
              \vrule height#1 width 0pt
               \vrule }%
        \hrule}%
}}}
\def\shuffl{{\mathchoice{\shuff{7pt}{3.5pt}}%
                        {\shuff{6pt}{3pt}}%
                        {\shuff{4pt}{2pt}}%
                        {\shuff{3pt}{1.5pt}}}}%
\def\shuffle{\, \shuffl \,}

\newcommand{\reduced}{{\operatorname{red}}}
\newcommand{\leaf}{\circ}
\newcommand{\PRT}{\mathcal{T}}
\newcommand{\type}[1]{\tau(#1)}
\newcommand{\tree}{\tau}

\newcommand{\dom}{{\operatorname{dom}}}
\newcommand{\im}{{\operatorname{im}}}

\newcommand{\id}{{\operatorname{id}}}

\newcommand{\aut}{{\operatorname{Aut}}}
\newcommand{\qsym}{{\operatorname{QSym}}}
\newcommand{\sym}{{\operatorname{Sym}}}
\newcommand{\lm}{{\operatorname{lm}}}
\newcommand{\hilbert}{{\mathcal H}}
\newcommand{\ideal}[1][I]{\mathcal #1}
\newcommand{\age}{{\mathcal A}}
\newcommand{\agealgebra}{{\K\mathcal A}}
\newcommand{\profile}{\varphi}
\newcommand{\lex}{{\operatorname{lex}}}
\newcommand{\revlex}{{\operatorname{revlex}}}
\newcommand{\degrevlex}{{\operatorname{degrevlex}}}
\newcommand{\setalgebra}[1][E]{{\K^{[#1]^{<\omega}}}}
\newcommand{\minmorphdec}{{\mathcal P}}

\newcommand{\goodsubset}{{\mathcal{C}}}
\newcommand{\setpartitions}{{\mathcal{S}}}
\newcommand{\coarser}{\preceq}

\newcommand{\R}{\mathbb{R}}
\newcommand{\K}{\mathbb{K}}
\newcommand{\N}{\mathbb{N}}

\newcommand{\Z}{\mathbb{Z}}
\newcommand{\Q}{\mathbb{Q}}

\newtheorem{theorem}{Theorem}[section]

\newtheorem{axiom}[theorem]{Axiom}
\newtheorem{axioms}[theorem]{Axioms}
\newtheorem{lemma}[theorem]{Lemma}

\newtheorem{proposition}[theorem]{Proposition} 
\newtheorem{corollary}[theorem]{Corollary} 
 
\theoremstyle{definition}

\newtheorem{problem}[theorem]{Problem}
\newtheorem{problems}[theorem]{Problems}
\newtheorem{conjecture}[theorem]{Conjecture}

\theoremstyle{remark}

\newtheorem{question}[theorem]{Question}

\newtheorem{example}[theorem]{Example}
\newtheorem{examples}[theorem]{Examples}

\def\centerpicture #1 by #2 (#3){\leavevmode
        \vbox to #2{
        \hrule width #1 height 0pt depth 0pt
        \vfill
        \special{pictfile #3}}}

\makeatletter
\newskip\@bigflushglue \@bigflushglue = -100pt plus 1fil

\def\bigcentering{\let\\\@centercr\rightskip\@bigflushglue%
\leftskip\@bigflushglue
\parindent\z@\parfillskip\z@skip}

\makeatother

\begin{document}
\maketitle

\begin{abstract}
  The \emph{profile} of a relational structure $R$ is the function
  $\profile_R$ which counts for every integer $n$ the number
  $\profile_R(n)$, possibly infinite, of substructures of $R$ induced
  on the $n$-element subsets, isomorphic substructures being
  identified.  If $\profile_R$ takes only finite values, this is the
  Hilbert function of a graded algebra associated with $R$, the
  \emph{age algebra} $\agealgebra(R)$, introduced by P.~J.~Cameron.
  
  In this paper we give a closer look at
  this association, particularly when the relational structure $R$
  admits a \emph{finite monomorphic decomposition}. This setting still
  encompass well-studied graded commutative algebras like invariant
  rings of finite permutation groups, or the rings of quasi-symmetric
  polynomials. We prove that $\profile_R$ is eventually a
  quasi-polynomial, this supporting the conjecture that, under mild
  assumptions on $R$, $\profile_R$ is eventually a quasi-polynomial
  when it is bounded by some polynomial.

  \medskip \noindent {\bf Keywords:} Relational structure, profile,
  graded algebra, Hilbert function, Hilbert series, polynomial growth.
\end{abstract}

\section*{Introduction}
This paper is  about a counting function: the
\emph{profile} of a relational structure $R$ and its
interplay with a graded connected commutative algebra associated with
$R$, the \emph{age algebra} of $R$.  

Many natural counting functions are profiles. Several interesting examples
come from permutation groups. For example, if $G$ is a permutation
group on a set $E$, the function $\theta_G$ which counts for every
integer $n$ the number of orbits of the action of $G$ on the
$n$-element subsets of $E$, is a profile, the \emph{orbital profile}
of $G$. Groups whose
orbital profile takes only finite values are called
\emph{oligomorphic}; their study, introduced by Cameron, is a whole
research subject by itself~\cite{Cameron.1990,Cameron.OPG.2009}.  If
$G$ acts on $\{1,\dots,k\}$, the Hilbert function of the subalgebra
$\K[X]^G$ of the polynomials in $\K[X]:= \K[X_1,\ldots, X_k]$ which
are invariant under the action of $G$ is a profile, and in fact an
orbital profile.
This fact led   Cameron  to associate a graded algebra $\agealgebra(R)$ to each
relational structure $R$~\cite{cameron.1997}; its main feature is that
its Hilbert function coincides with the profile of $R$ as long as it
takes only finite values. As it is well know, the Hilbert function of
a graded commutative algebra $A$ is eventually a quasi-polynomial (hence bounded by some polynomial), provided that $A$ is finitely
generated. The converse does not hold.  Still, this leads us to
conjecture that the profile $\profile_R$ of a relational structure $R$
is eventually a quasi-polynomial when $\profile_R$ is bounded by some
polynomial (and the kernel of $R$ is finite).

This conjecture holds if $R$ is an undirected
graph~\cite{Balogh_Bollobas_Saks_Sos.2009} or a
tournament~\cite{Boudabbous.Pouzet.2009}.  In this paper, we prove
that it holds for any relational structure $R$ admitting a finite
\emph{monomorphic decomposition}. Its age algebra is (essentially) a
graded subalgebra of a finitely generated polynomial algebra. It needs
not be finitely generated (see
Example~\ref{example.notfinitelygenerated}). Still, and this is our
main result, the profile is eventually a quasi-polynomial whose degree
is controlled by the dimension of the monomorphic decomposition of $R$
(Theorem~\ref{theorem.quasipolynomial}). This result was applied
in~\cite{Boudabbous.Pouzet.2009} to show that the above conjecture
holds for tournaments.

Relational structures admitting a finite monomorphic decomposition are
not so peculiar. Many familiar algebras, like invariant rings of
permutation groups, rings of quasi-symmetric polynomials, can be
realized this way.  We give many examples in
Appendix~\ref{appendix.examples}. Further studies, notably a
characterization of these relational structures for which the age
algebra is finitely generated are included
in~\cite{Pouzet_Thiery.AgeAlgebra2}.

The study of the profile started in the seventies;
see~\cite{Pouzet.2006.SurveyProfile} for a survey
and~\cite{Pouzet.2008.IntegralDomain} for more recent results. Our
line of work is parallel to the numerous researches made in recent
years about the behavior of counting functions for hereditary classes
made of finite structures, like undirected graphs, posets,
tournaments, ordered graphs, or permutations, which also enter into
this frame; see~\cite{Klazar.2008.Overview}
and~\cite{Bollobas.1998.HereditaryPropertiesOfGraphs} for a survey,
and~\cite{Balogh_Bollobas_Morris.2007,Balogh_Bollobas_Saks_Sos.2009,Marcus_Tardos.2004,Albert_Atkinson.2005,Albert_Atkinson_Brignall.2007,Vatter.2011,Kaiser_Klazar.2003,Klazar.2008,Brignall_Huczynska_Vatter.2008,Oudrar.Pouzet.2011}. These
classes are \emph{hereditary} in the sense that they contain all
induced structures of each of their members; in several instances,
members of these classes are counted up to isomorphism and with
respect to their size.

Results point out jumps in the behavior of these counting
functions. Such jumps were for example announced for extensive
hereditary classes in~\cite {Pouzet.1980}, with a proof for the jump
from constant to linear published in~\cite{Pouzet.RPE.1981}.  The
growth is typically polynomial or faster than any polynomial, though
not necessarily exponential (as indicates the partition function; see
Example~\ref{example.path}).  For example, the growth of an hereditary
class of graphs is either polynomial or faster than the partition
function~\cite{Balogh_Bollobas_Saks_Sos.2009}. In several instances,
these counting functions are eventually quasi-polynomials,
e.g.~\cite{Balogh_Bollobas_Saks_Sos.2009} for graphs
and~\cite{Kaiser_Klazar.2003} for permutations.

Klazar asked in his survey \cite{Klazar.2008.Overview} how the two
approaches relate. At first glance, the structures we consider are
more general; however the classes we consider in this paper are more
restrictive: the profile of a relational structure $R$ is the counting
function of its age, and ages are hereditary classes structures which
are up-directed by embeddability (\cite{Fraisse.1954}). A priori,
results on the behavior of counting functions of ages do not extend
straightforwardly to hereditary classes. There is a notable exception:
the study of hereditary classes with polynomially bounded profiles can
indeed be reduced to that of ages, thanks to the following result:
\begin{theorem}\label{hereditary/ideals}
  Consider an hereditary class $\mathfrak C$ of finite structures of
  fixed finite signature. If $\mathfrak C$ has polynomially bounded
  profile then it is a finite union of ages. Otherwise it contains an
  age with non polynomially bounded profile.
\end{theorem}
 
 The  proof is given in section~\ref{section.lemma}.
 
This paper is organized as follows. In Section~\ref{section.results}
we recall the definitions and basic properties of relational
structures, their profiles, and age algebras, and state our guiding
problems. We introduce the key combinatorial notion of
\emph{monomorphic decomposition} of a relational structure, mention
the existence of a unique minimal one
(Proposition~\ref{proposition.monodec}), and state our main theorem
(Theorem~\ref{theorem.quasipolynomial}) together with some other results. In
Section~\ref{section.monomorphic_decompositions}, we study the
properties of monomorphic decompositions and prove Proposition~\ref{proposition.monodec},
while Section~\ref{section.quasipolynomial} contains the proof of
Theorem~\ref{theorem.quasipolynomial}. Analyzing lots of examples has
been an essential tool in this exploration.
Appendix~\ref{appendix.examples} gathers them, with a description of
their age algebras; see in particular
Examples~\ref{example.notfinitelygenerated}, \ref{example.qsym}, and
\ref{example.nonCM.groupoid}, and
Proposition~\ref{proposition.planar_shuffle_algebra}.
We urge the reader to start by browsing them, and to
come back to them each time a new notion is introduced.
The overview of
the results in Table~\ref{table.overview} (which includes results
from~\cite{Pouzet_Thiery.AgeAlgebra2}) may be of help as well.

\subsection*{Acknowledgments}

A strong impulse to this research came from a paper by Garsia and
Wallach~\cite{Garsia_Wallach.2003} who proved that the ring of
quasi-symmetric polynomials is Cohen-Macaulay. Most, but not all, of
the results where previously announced at the FPSAC conference in
honor of Adriano Garsia~\cite{Pouzet_Thiery.AgeAlgebra.2005}, at the
14th Symposium of the Tunisian Mathematical Society, held in Hammamet
in 2006~\cite{Pouzet.2006.SurveyProfile}, and at
ROGICS'08~\cite{Pouzet_Thiery.AgeAlgebra.2008}.

We would like to thank the anonymous referee for a thorough report
including many helpful suggestions for improvements.

\begin{sidewaystable}
  \small
  \begin{bigcenter}
    \hspace{.5cm} %
    \makeatletter
\newcommand\myvcentermiddle[1]{%
  \setbox0\hbox{#1}%
  \dimen255=.5\dp0%
  \advance\dimen255 by -.5\ht0%
  \advance\dimen255 by .5ex%
  \ifvmode\hskip0em\fi\raise\dimen255\box0}
\makeatother
%
%
\newcommand{\includegraphicsrelation}[1]{
  \myvcentermiddle{\includegraphics[height=10ex]{PICTURES/#1}}}
\newcommand{\tabtitle}[2]{\parbox[m]{#1}{\begin{center}#2\end{center}}}
\newcommand{\polyring}[2]{#1^{\makebox[0cm][l]{$\scriptscriptstyle\
      \leftarrow\in#2$}}}
\renewcommand{\footnoterule}{}
\renewcommand\thefootnote{\alph{footnote}}
\begin{tabular}{|m{25ex}|*{9}{@{\hspace{0.3ex}}c@{\hspace{0.3ex}}|}}
  \hline
  & 
    \tabtitle{10.5ex}{Relational structure} &
      \tabtitle{14ex}{\begin{center}local \mbox{isomorphisms}\end{center}} &
        \tabtitle{8ex}{Hilbert Series} &
          \tabtitle{10.5ex}{Finitely generated} &
            \tabtitle{7ex}{Degree bound} &
              \tabtitle{11ex}{Krull \mbox{dimension}} &
                \tabtitle{8ex}{Sym-module} &
                  \tabtitle{10ex}{Cohen-Macaulay} &
                    \tabtitle{8ex}{Finite SAGBI}
\\\hline
  $|X|<\infty$ &&&
        $\frac{\polyring{P(Z)}{\Z[Z]}}{(1-Z)\cdots(1-Z^{|X_\infty|})}$ &
          never\footnotemark[1]\footnotetext[1]{\tiny unless when stated otherwise
      below} &
            $=\infty$\footnotemark[1] &
	      $\leq |X_\infty|$ &
                no\footnotemark[2]\footnotetext[2]{\tiny here, "no" means
      "not always": there are both examples and counter-examples} &
		  no\footnotemark[2] &
		    no\footnotemark[2]\\\hline
  Hereditary optimal &&&
        $\frac{\polyring{P(Z)}{\Z[Z]}}{(1-Z)\cdots(1-Z^{|X_\infty|})}$ &
          yes &
            $<\infty$ &
	      $|X_\infty|$ &
                almost &
		  no\footnotemark[2] &
                    no\footnotemark[2]\\\hline
  Shape preserving &&&
        $\frac{\polyring{P(Z)}{\Z[Z]}}{(1-Z)\cdots(1-Z^{|X_\infty|})}$ &
          yes &
            $<\infty$ &
	      $|X_\infty|$ &
                yes &
		  no\footnotemark[2] &
                    no\footnotemark[2]\\\hline
  r-Quasi symmetric\vfil\mbox{polynomials}~\cite{Hivert.RQSym.2004} &
    \includegraphicsrelation{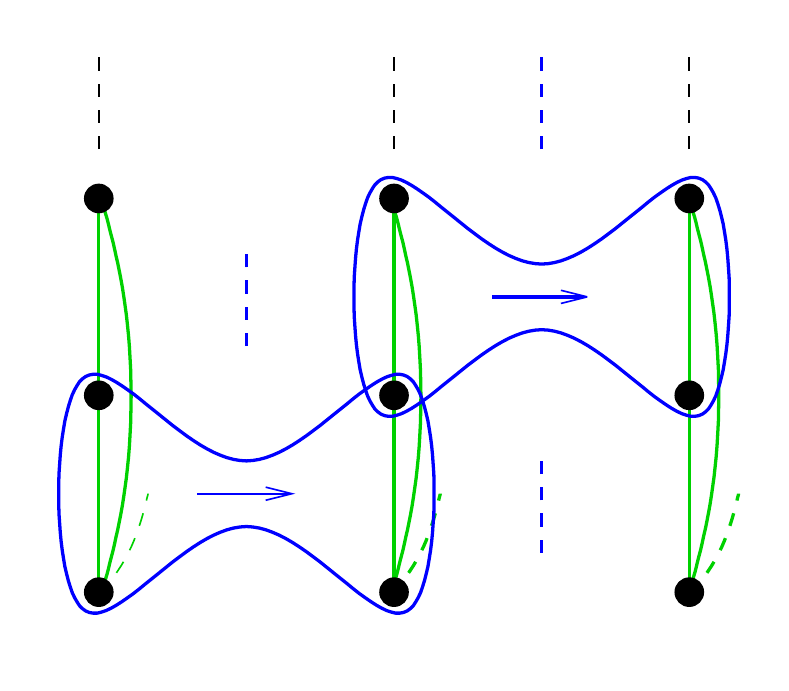} &&
        $\frac{\polyring{P(Z)}{\N[Z]}}{(1-Z)\cdots(1-Z^{|X|})}$ &
          yes &
            $\leq\frac{|X|(|X|+2r-1)}{2}$ &
	      $|X|$ &
                yes &
		  yes &
                    no\\\hline
  Invariants of a \vfil\mbox{permutation groupoid $G$} &&
      $G\wr \sg_\N$&
        $\frac{\polyring{P(Z)}{\Z[Z]}}{(1-Z)\cdots(1-Z^{|X|})}$ &
          yes &
            $\leq\frac{|X|(|X|+1)}{2}$ &
	      $|X|$ &
                yes &
		  no\footnotemark[2] &
                    \tabtitle{10ex}{never\footnotemark[1]}\\\hline
  Non Cohen-Macaulay\vfil\mbox{example}&
    \includegraphicsrelation{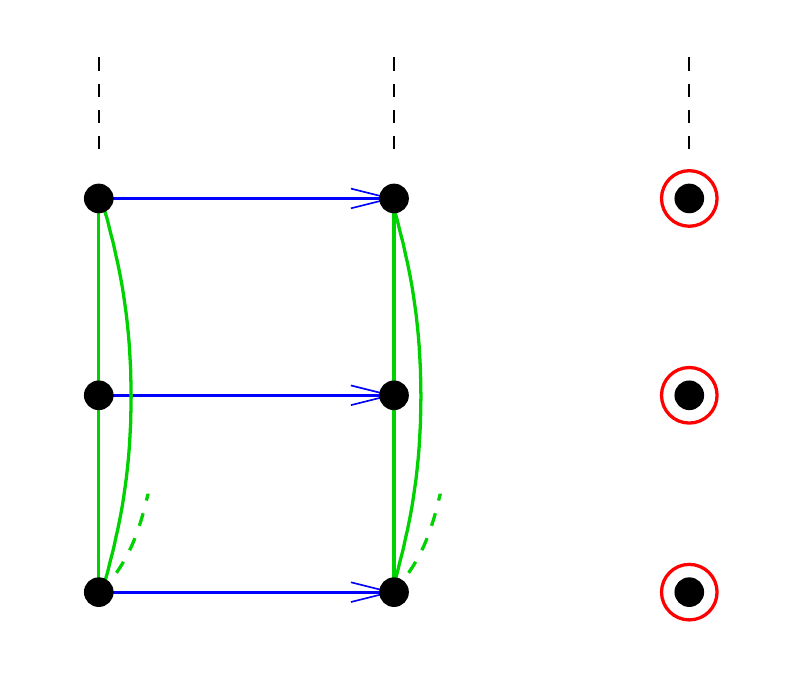} &
      $\langle 1\mapsto 2\rangle \wr \sg_\N$ &
        $\frac{1+Z^2+Z^3- Z^4}{(1-Z)^2(1-Z^2)}$ &
          yes &
            $2$ &
	      $|X|$ &
                yes &
		  no &
                    no \\\hline
  Quasi symmetric\vfil\mbox{polynomials}~\cite{Gessel.QSym.1984} &
    \includegraphicsrelation{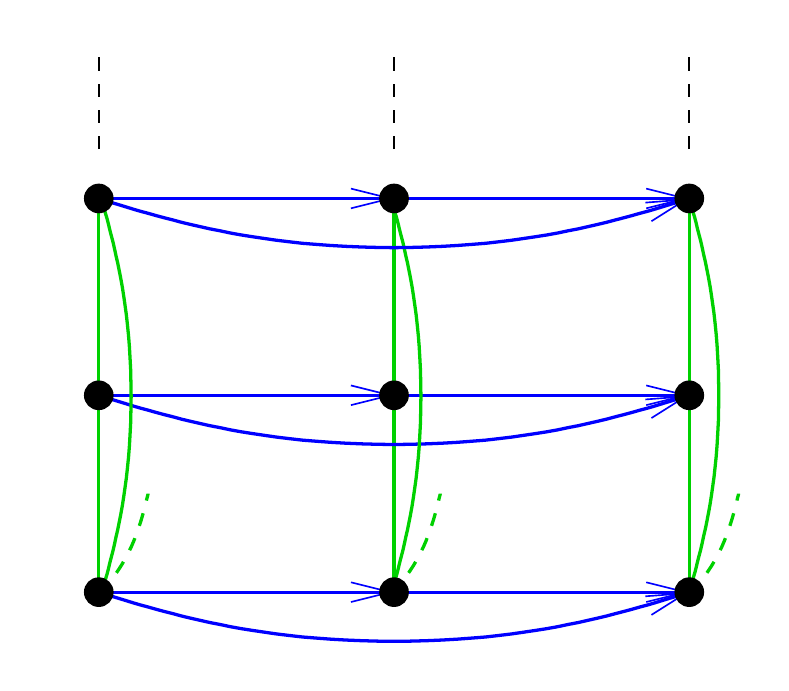} &
      $\operatorname{Inc} \wr \sg_\N$ &
        $\frac{\polyring{P(Z)}{\N[Z]}}{(1-Z)\cdots(1-Z^{|X|})}$ &
          yes&
            $\leq\frac{|X|(|X|+1)}{2}$ &
	      $|X|$ &
                yes &
		  yes~\cite{Garsia_Wallach.2003} &
		    no\\\hline
  Invariants of a\vfil\mbox{permutation group $G$} &&
      $G\wr \sg_\N$&
        $\frac{\polyring{P(Z)}{\in\N[Z]}}{(1-Z)\cdots(1-Z^{|X|})}$ &
          yes &
            $\leq\frac{|X|(|X|-1)}{2}$ &
	      $|X|$ &
                yes &
		  yes &
		    \tabtitle{10ex}{never\footnotemark[1] \cite{Thiery_Thomasse.SAGBI.2002}}\\\hline
  Symmetric\vfil\mbox{polynomials} &
    \includegraphicsrelation{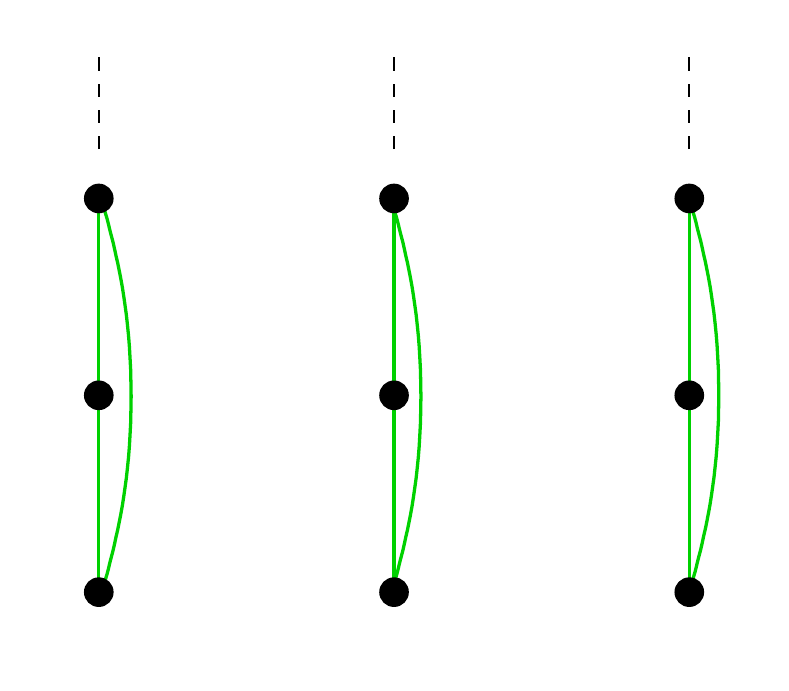} &
      $\sg_n\wr \sg_\N$&
        $\frac{1}{(1-Z)\cdots(1-Z^{|X|})}$ &
          yes &
            $|X|$ &
	      $|X|$ &
                yes &
		  yes &
		    yes\\\hline
  Polynomials &
    \includegraphicsrelation{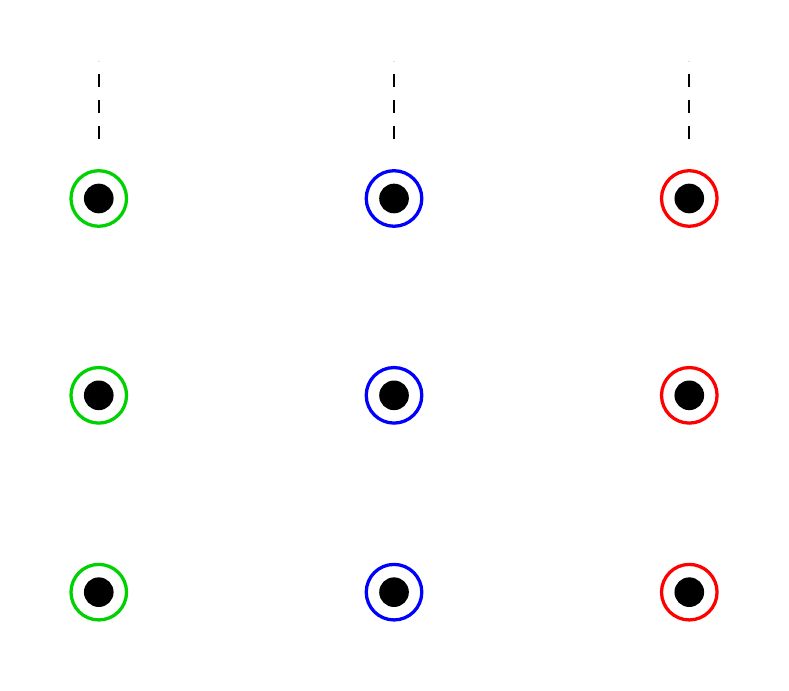} &
      $\id\wr \sg_\N$&
        $\frac{(1+Z)\cdots(1+Z+\dots+Z^{|X|})}{(1-Z)\cdots(1-Z^{|X|})}$ &
          yes &
            $1$ &
	      $|X|$ &
                yes &
		  yes &
		    yes\\\hline
\end{tabular}


  \end{bigcenter}
  \caption{Overview of the results}
  \label{table.overview}
\end{sidewaystable}

\section{Age algebras and quasi-polynomiality of relational structure
  admitting a finite monomorphic decomposition} \label{section.results}

\subsection{Relational structures and their profile}

A \emph{relational structure} is a realization of a language whose
non-logical symbols are predicates. This is a pair
$R:=(E,(\rho_{i})_{i\in I})$ made of a set $E$ and a family of
$m_i$-ary relations $\rho_i$ on $E$.  The set $E$ is the \emph{domain}
or \emph{base} of $R$; the family $\mu:= (m_i)_{i\in I}$ is the
\emph{signature} of $R$; the signature is \emph{finite} if $I$ is.
The \emph{substructure induced by $R$ on a
  subset $A$} of $E$, simply called the \emph{restriction of $R$ to
  $A$}, is the relational structure $R_{\restriction A}:= (A,
(A^{m_i}\cap \rho_i)_{i\in I})$.  The notion  of \emph{isomorphism} between relational structures 
is defined in the natural way. A
\emph{local isomorphism} of $R$ is any isomorphism between two restrictions
of $R$. Two relational structures $R$ and $R'$
are \emph{isomorphic} is there is an  isomorphism $f$ from $R$ onto
$R'$. We also say that they have the same \emph{isomorphism
  type}.
The \emph{isomorphism type} of a relational structure
is a formal object $\tau(R)$ such that  a relational structure $R'$ is
isomorphic to $R$ if and only if $\tau(R')=\tau (R)$. In some
situations, isomorphism types have a concrete representation. Let $R$
be as above. For two
subsets $A$ and $A'$ of $E$, we set $A\approx A'$ if $R_{\restriction
  A}$ and $R_{\restriction A'}$ are isomorphic. The \emph{orbit} of a
subset $A$ of $E$ is the set $\overline A$ of all subsets $A'$ of $E$
such that $A'\approx A$ (the name ``orbit'' is given by analogy with
the case of permutation groups; see Subsection~\ref{subsection:groups}).
The orbit $\overline A$ can play the role of the isomorphism type
$\type{A}:=\type{R_{\restriction A}}$ of $R_{\restriction A}$.

The \emph{profile} of $R$ is the function $\profile_R$ which counts,
for every integer $n$, the number $\profile_R(n)$ of isomorphism types
of restrictions of $R$ on $n$-element subsets.
Clearly, this
function only depends upon the set $\age(R)$ of isomorphism types of
finite restrictions of $R$; this set, called the \emph{age
  of $R$}, was introduced by R.~Fra{\"\i}ss{\'e}
(see~\cite{Fraisse.TR.2000}). If the signature of $R$ is finite, $\profile_R(n)$ is
necessarily finite. In order to capture examples coming from algebra
and group theory, we cannot preclude an infinite signature. However,
since the profile is finite in these examples, and unless explicitly
stated otherwise, \emph{we always make the assumption that
  $\profile_R(n)$ is finite, even if the signature of $R$ is not}.

The profile of an infinite relational structure is
non-decreasing. Furthermore, provided some mild conditions, there are
jumps in the behavior of the profile:
\begin{theorem}
  \label{profilpouzet1}
  Let $R := (E, (\rho_i)_{i \in I})$ be a relational structure on an
  infinite set. Then, $\profile_R$ is non-decreasing. And provided
  that either the signature $\mu$ is bounded or the kernel $K(R)$ of
  $R$ is finite, the growth of $\profile_R$ is either
  \emph{polynomial} or as fast as every polynomial.
\end{theorem}
A map $\profile : \N\rightarrow \N$ has {\it polynomial growth}, of
{\it degree} $k$, if $an^k\leq \profile(n)\leq bn^k $ for some $a,b>0$
and $n$ large enough. The \emph{kernel} of $R$ is the set $K(R)$ of $x
\in E$ such that $\age(R_{\restriction E \setminus \{x\}}) \neq \age(R)$.
Relations with empty kernel are the \emph{age-inexhaustible} relations of
R.~Fra\"{\i}ss\'e'(see~\cite{Fraisse.TR.2000}). We call \emph{almost
age-inexhaustible} those with finite kernel.

The hypothesis about the kernel is not ad hoc. As it turns out, if the
growth of the profile of a relational structure with a bounded
signature is bounded by a polynomial then its kernel is finite.  Some
hypotheses on $R$ are needed, indeed \emph{for every increasing and
  unbounded map $\varphi: \N \rightarrow \N$, there is a relational
  structure $R$ such that $\profile_R$ is unbounded and eventually
  bounded above by $\varphi$} (see~\cite{Pouzet.RPE.1981}).  The
first part of the result was obtained in $1971$ by the first author
(see Exercise~8 p.~113~\cite{Fraisse.CLM1.1971}). A proof based on
linear algebra is given in~\cite{Pouzet.1976}. The second part was
obtained in \cite{Pouzet.TR.1978} and a part was published in
\cite{Pouzet.RPE.1981}.

The theorem above is not the best possible. It is natural to ask
\begin{question}
  \label{question1}
  Does $\profile_R$ have \emph{polynomial growth} in the strong sense:
  $\profile_R(n)\sim a n^{k}$ for some positive real $a$ whenever $R$
  has bounded signature or finite kernel and $\profile_R$ is bounded
  above by some polynomial.
\end{question}
This question, raised by Cameron \cite{Cameron.1990} in the special
case of orbital profile, is unsettled.

We consider a stronger condition. Recall that a map $\varphi:
\N\rightarrow \N$ is \emph{a quasi-polynomial} of degree $k$ if
$\varphi(n)=a_{k}(n)n^{k}+\cdots+ a_0(n)$ whose coefficients
$a_{k}(n), \dots, a_0(n)$ are periodic functions.
Note that, \emph{when a profile is eventually a quasi-polynomial, it
  has polynomial growth in the strong sense}. Indeed, since the
profile is non-decreasing, the coefficient $a_{k}(n)$ of highest
degree of the quasipolynomial is constant. We make the following:
\begin{conjecture}
  \label{conjecture.quasipolynomial}
  The profile of a relational structure with bounded signature or
  finite kernel is eventually a quasi-polynomial whenever the profile
  is bounded by some polynomial.
\end{conjecture}

This conjecture is motivated by the association made by Cameron of a
graded commutative algebra, the age algebra $\agealgebra(R)$, to a
relational structure $R$.  Indeed, if $A$ is a graded commutative
algebra, the \emph{Hilbert function} $h_A$ of $A$, where $h_A(n)$ is
the dimension of the homogeneous component of degree $n$ of $A$, is
eventually a quasi-polynomial whenever $A$ is finitely generated. In
fact, if $A$ is such an algebra, the generating series
$\hilbert_A(Z):=\sum_n h_A(n) Z^n$, called the \emph{Hilbert series}
of $A$, is a rational fraction of the form
\begin{equation}
  \label{equation.fraction.finitelyGenerated}
  \frac{P(Z)}{(1-Z^{n_1})(1-Z^{n_2})\cdots(1-Z^{n_k})}\ ,
\end{equation}
where $1=n_1\leq \cdots \leq n_k$ and $P(Z)\in\Z[Z]$ (see
e.g.~\cite[Chapter 9, \S 2]{Cox_al.IVA}).
Furthermore, whenever $A$ is Cohen-Macaulay, $P(Z)\in\N[Z]$ for some
choice of $n_1,\dots,n_k$. And, as it is also well known
(\cite{Stanley.1999.EnumerativeCombinatorics1}), a counting function
is eventually a quasi-polynomial of degree at most $k-1$ whenever its
generating series has the form
(\ref{equation.fraction.finitelyGenerated}) above.

As shown by Cameron, the Hilbert function of the age algebra of a
relational structure $R$ is the profile of $R$ (provided that it takes
only finite values). Thus, provided that the age algebra is finitely
generated, Conjecture~\ref{conjecture.quasipolynomial} and hence
Question~\ref{question1} have a positive answer. In fact, as we will
see, there are many relational structures for which the generating
series of their profile has the form
(\ref{equation.fraction.finitelyGenerated}) above, but few for which the age algebra is finitely generated.

We present now the age algebra.

\subsection{Age algebras}

Let $\K$ be a field of characteristic $0$, and $E$ be a set. For $n\geq
0$, denote by $[E]^n$ the set of the subsets of $E$ of size $n$, and let
$\K^{[E]^{n}}$ be the vector space of maps $f: [E]^n \rightarrow \K$. The
\emph{set algebra} is the graded connected commutative algebra
$\setalgebra:=\bigoplus_n \K^{[E]^{n}}$, where the product of $f:
[E]^m\rightarrow \K$ and $g: [E]^n\rightarrow \K$ is defined as $fg:
[E]^{m+n}\rightarrow \K$ such that:
\begin{equation}
  \label{eq:product}
  (fg)(A):= \sum_{(A_1,A_2) \suchthat  A=A_1\uplus A_2}  f(A_1)g(A_2) \,.
\end{equation}
Identifying a set $S$ with its characteristic function $\chi_S$,
elements of the set algebra can be thought as (possibly infinite but of
bounded degree) linear combination of sets; the unit is the empty set,
and the product of two sets is their disjoint union, or $0$ if their
intersection is non trivial.

Let $R$ be a relational structure with base set $E$. A map $f:
[E]^m\rightarrow \K$ is \emph{$R$-invariant} if $f(A)=f(A')$ whenever
$A\approx A'$. It is easy to show that the product of two
$R$-invariant maps is again invariant. The $\K$-vector space spanned
by the $R$-invariant maps is therefore a graded connected commutative
subalgebra of the set algebra, the \emph{age algebra of $R$}, that we
denote by $\agealgebra(R)$. It can be shown that two relational
structures with the same age yield the same algebra (up to an
isomorphism of graded algebras); thus the name, coined by Cameron who
invented the notion~\cite{cameron.1997}.  If the profile of $R$ takes
only finite values, then $\agealgebra(R)$ identifies with the set of
(finite) linear combinations of elements of $\age(R)$ and, as pointed
out by Cameron, $\profile_R(n)$ is the dimension of the homogeneous
component of degree $n$ of $\agealgebra(R)$; indeed, define an orbit sum as the characteristic function of an orbit; more specifically,  the
\emph{orbit sum} of an isomorphism type $\tau\in \age(R)$ is the
characteristic function $o_\tau := \sum_{A\in [E]^{<\omega} \suchthat
  \tau(A)= \tau} A$ of its representatives in $R$; then observe
that the set of orbit sums form a basis of the age algebra
$\agealgebra(R)$). By a slight abuse, we sometimes identify $\tau$
with its orbit sum to see it as an element of $\agealgebra(R)$.

Given three isomorphic types $\tau, \tau_1,\tau_2$, we define a
coefficient $c_{\tau_1,\tau_2}^\tau$ by taking any subset $A$ of $E$
of type $\tau$ and setting
\begin{equation}
  c_{\tau_1,\tau_2}^\tau := |\{ (A_1,A_2) \suchthat A_1\uplus A_2=A,
  \type{A_1}=\tau_1, \type{A_2}=\tau_2\}|\,.
\end{equation}
Clearly, this coefficient does not depend on the choice of $A$.
The collection $(c_{\tau_1,\tau_2}^\tau)_{\tau,\tau_1,\tau_2}$ are the
\emph{structure constants} of the age algebra:
\begin{equation}
  o_{\tau_1} o_{\tau_2} = \sum_{\tau} c_{\tau_1,\tau_2}^\tau o_{\tau}\,.
\end{equation}

Let us illustrates the role that the age algebra of Cameron can play.
Let $e:= \sum_{a\in E} \{a\}$, which we can think of as the sum of
isomorphic types of the one-element restrictions of $R$. Let $U$ be
the graded free algebra $\K[e]=\bigoplus_{n=0}^\infty \K e^n$.
Cameron (see~\cite{cameron.1997}) proved:
\begin{theorem}
  \label{theorem.cameron}
  If $R$ is infinite then $e$ is not a zero-divisor; namely for any
  $u\in \agealgebra(R)$, $e u=0$ if and only if $u=0$.
\end{theorem}
This result implies that $\profile_R$ is non decreasing. Indeed, the
image of a basis of the vector space $\agealgebra(R)_n$ under
multiplication by $e$ is a linearly independent subset of
$\agealgebra(R)_{n+1}$.

The relationship between profile and age algebra is particularly
simple for relational structures with bounded profile. These
structures were characterized in~\cite{Fraisse_Pouzet.1971} for finite
signature and in~\cite[Th\'eor\`eme 1.2]{Pouzet.RPE.1981} for
arbitrary signature, by means of Ramsey theorem. We recall it below.

Let $R:= (E, (\rho_i)_{i \in I})$ be a relational structure, $F$
be a subset of $E$ and $B:=E\setminus F$. Following R.~Fra\"{\i}ss\'e, who invented these notions,  we say that $R$ is
$F$-\emph{monomorphic} if for every integers $n$ and every $A,A'\in
[B]^n$ there is an isomorphism from $R_{\restriction A\cup F}$
onto $R_{\restriction A'\cup F}$ which is the identity on $F$. We say that
$R$ is $F$-{\it chainable} if there is a linear order $\leq$ on $B$
such that the local isomorphisms of $(B, \leq)$ extended by the
identity on $F$ are local isomorphisms of $R$. We say that $R$ is
\emph{almost monomorphic}, resp. \emph{almost chainable},  if $R$ is
$F$-monomorphic, resp. $F$-chainable,  for some finite subset $F$ of
$E$. We say that $R$ is \emph{monomorphic}, resp.  \emph{chainable}, if
one can take $F$ empty.

Assume for example, that $R$ is made of a single relation $\rho$ and
is chainable by some linear order $\leq$. If $\rho$ is unary, then it
is the full or empty relation. If $\rho$ is binary, there are eight
possibilities: four if $\rho$ is reflexive ($\rho$ coincides with
either $\leq$, its opposite $\geq$, the equality relation $=$, or the
complete relation $E\times E$) and four if not (the same as for the
reflexive case, with loops removed).

For infinite relational structures, the notions of monomorphy and
chainability coincide~\cite{Fraisse.1954}. For finite structures, they
are distinct. However, it was proved by C.~Frasnay that, for any integer
$n$, there is an integer $f(n)$ such that every monomorphic relational
structure of arity at most $n$ and size at least $f(n)$ is chainable
(see~\cite{Frasnay.1965} and \cite[Chapter 13]{Fraisse.TR.2000}).

The following theorem links the profile with the age algebra in the
context of almost-chainable and almost-monomorphic relational
structures.
\begin{theorem}\label{theorem.boundedProfile}
  Let $R$ be a relational structure with $E$ infinite. Then, the
  following properties are equivalent:
  \begin{enumerate}
  \item[(a)] The profile of $R$ is bounded.
  \item[(b)] $R$ is almost-monomorphic.
  \item[(c)] $R$ is almost-chainable.
  \item[(d)] The Hilbert series is of the following form, with
    $P(Z)\in \N[Z]$ and $P(1)\ne 0$:
    \begin{displaymath}
      \hilbert_R=\frac {P(Z)} {1-Z}\ .
    \end{displaymath}
  \item[(e)] The age algebra is a finite dimensional free-module over
    the free-algebra $\K[e]$, where $e:=\sum_{a\in E} a$; in
    particular it is finitely generated and Cohen-Macaulay.
  \end{enumerate}
\end{theorem}
\begin{proof}
  Trivially, (c) implies (b) and (b) implies (a). The equivalence
  between (a) and (c) is in~\cite{Fraisse_Pouzet.1971} for finite
  signature and in~\cite[1.2 Th\'eor\`eme, p.~317]{Pouzet.RPE.1981}
  for arbitrary signature.

  Straightforwardly, (e) implies (d) and (d) implies (a).  Finally (e)
  follows from (a): indeed, by Theorem~\ref{theorem.cameron}, $e$ is
  not a zero divisor in $\agealgebra(R)$; using the grading, it
  follows that $\K[e]$ is a free algebra and that $\agealgebra(R)$ is
  a free-module over $\K[e]$. Since the profile is bounded, this
  free-module is finite dimensional.
\end{proof}

\subsection{Relational structures admitting a finite monomorphic decomposition}
\label{subsection.finiteMonomorphicDecomposition}

We now introduce a combinatorial notion which generalizes that of
almost-monomorphic relational structure.

Let $E$ be a set and $(E_x)_{x\in X}$ be a set partition of $E$. Write
$X_\infty:=\{x\in X\suchthat |E_x|=\infty\}$. For a subset $A$ of $E$,
set $d_x(A):=|A\cap E_x|$, $d(A):=(d_x(A))_{x\in X}$ and $\overline
d(A)$ be the sequence $d(A)$ sorted in decreasing order. The sequence
$\overline d(A)$ is the \emph{shape} of $A$ with respect to the
partition $(E_x)_{x\in X}$.  When $A$ is finite it is often
convenient, and even meaningful, to encode $d(A)$ as a monomial in
$\K[X]$, namely
$X^A := \prod_{x\in X} x^{d_x(A)}$; %
furthermore, up to trailing zero parts, $\overline d(A)$ is an integer
partition of $|A|$.

Let $R$ be a relational structure on $E$.  We call $(E_x)_{x\in X}$ a
\emph{monomorphic decomposition} of $R$ if the induced structures on
two finite subsets $A$ and $A'$ of $E$ are isomorphic whenever $X^A =
X^{A'}$. The following proposition states the basic result about
monomorphic decompositions; we will prove it in a slightly more
general setting (see
Proposition~\ref{proposition.lattice_monomorphic_decomposition}).
\begin{proposition}
  \label{proposition.monodec}
  There is a monomorphic decomposition of $R$ of which every other
  monomorphic decomposition of $R$ is a refinement.
\end{proposition}
This monomorphic decomposition is called \emph{minimal}; its number
$k:=k(R):=|X_\infty|$ of infinite blocks is the \emph{monomorphic
  dimension} of $R$.

Our main result, which we prove in
Section~\ref{section.quasipolynomial}, is a complete solution for
Conjecture~\ref{conjecture.quasipolynomial} for relational structures
admitting a finite monomorphic decomposition (the almost-monomorphic
ones being those of dimension $1$):
\begin{theorem}\label{theorem.quasipolynomial}
  Let $R$ be an infinite relational structure with a finite
  monomorphic decomposition, and let $k$ be its monomorphic dimension.
  Then, the generating series $\hilbert_R$ is a rational fraction of
  the following form, with $P\in\Z[Z]$ and $P(1)\ne 0$:
  \begin{displaymath}
    \frac{P(Z)}{(1-Z)(1-Z^2)\cdots(1-Z^k)}\ .
  \end{displaymath}
  In particular, $\profile_R$ is eventually a quasi polynomial of
  degree $k-1$, hence $\profile_R(n)\sim an^{k-1}$.
\end{theorem}
This refines Theorem~2.16 of~\cite{Pouzet_Thiery.AgeAlgebra.2005}, by
refining the denominator and the growth rate.

Among relational structures which admit a finite monomorphic
decomposition, those made of a single unary or binary relation are
easy to characterize (cf. Corollary~\ref{cor:binary}). For example,
the undirected graphs admitting a finite monomorphic decomposition are
the lexicographical sums of cliques or independent sets indexed by a
finite graph. Similarly, the tournaments admitting a finite
monomorphic decomposition are the lexicographical sums of acyclic (aka
transitive) tournaments indexed by a finite tournament.
In Section~\ref{thm:charcfinitemonomorph.proof} we prove the following generalization.
\begin{theorem}
 \label{thm:charcfinitemonomorph}
 A relational structure $R:=(E, (\rho_{i})_{i\in I})$ admits a finite
 monomorphic decomposition if and only if there exists a linear order $\leq$ on $E$ and a finite
 partition $(E_x)_{x\in X}$ of $E$ into intervals of $(E, \leq)$ such that every local isomorphism of $(E, \leq)$ which preserves each interval is a local isomorphism of $R$. Moreover,
 there exists such a partition whose number of infinite blocks is the
 monomorphic dimension of $R$.
\end{theorem}

The interest for this notion goes beyond these examples and this
characterization.  It turns out that familiar algebras like invariant
rings of finite permutation groups can be realized as age algebras of
relational structures admitting a finite monomorphic decomposition
(see Example~\ref{example.permgroup}); another example is the algebra
of quasi-symmetric polynomials and variants (see
Examples~\ref{example.qsym} of Appendix~\ref{appendix.examples}).

The relationship with polynomials is not accidental:
\begin{proposition}[\cite{Pouzet_Thiery.AgeAlgebra2}]
  If $R$ admits a monomorphic decomposition into finitely many blocks
  $E_1,\dots, E_k$, all infinite, then the age algebra $\agealgebra(R)$ is
  isomorphic to a subalgebra $\K[x_1, \dots, x_k]^R$ of the algebra
  $\K[x_1, \dots, x_k]$ of polynomials in the indeterminates $x_1,
  \dots, x_k$.
\end{proposition}

Invariant rings of finite permutation groups, as well as the algebras
of quasi-symmetric polynomials are finitely generated and in fact
Cohen-Macaulay.
It is however worth noticing on the onset that there are examples of
relational structures such that
\begin{itemize}
\item $X$ is finite, but the age algebra is not finitely generated
  (see Example~\ref{example.notfinitelygenerated});
\item $X$ is finite, the Hilbert series is of the form of
  Equation~\ref{equation.fraction.finitelyGenerated} with
  $P(Z)\in\N[Z]$ but the age algebra is not
  Cohen-Macaulay (see Example~\ref{example.nonCM.groupoid});
\item $X$ is infinite but the profile still has polynomial growth (see
  Examples~\ref{example.QwreathG} and~\ref{example.notFinitelyGeneratedInfiniteDecomposition}).
\end{itemize}
In those examples, the profile is still a quasi-polynomial. This
raises the following problems.
\begin{problems}
  \label{problems}
  Let $R$ be a relational structure whose profile is bounded by some
  polynomial. Find combinatorial conditions on $R$ for
  \begin{itemize}
  \item[(a)] the profile to be eventually a quasi-polynomial;
  \item[(b)] the age algebra to be finitely generated;
  \item[(c)] the age algebra to be Cohen-Macaulay.
  \end{itemize}
\end{problems}
We give here a partial answer to (a), while (b) and (c) are
investigated in length in a subsequent
paper~\cite{Pouzet_Thiery.AgeAlgebra2}.

\section{Monomorphic decompositions}
\label{section.monomorphic_decompositions}

We start by stating some simple properties of set partitions. For the
sake of completeness, proofs are included.  Then, we apply those
results
to the lattice of monomorphic decompositions
(Proposition~\ref{proposition.lattice_monomorphic_decomposition}), and derive
a lower bound on the profile (Theorem~\ref{theorem.minimalGrowthRate})

\subsection{On certain lattices of set partitions}
\label{section.good_set_partitions}

Fix a set $E$, finite or not, and a collection $\goodsubset$ of
subsets of $E$ which are called \emph{good}. We assume that this
collection satisfies the following:
\begin{axioms}[Goodness axioms]
  \label{axioms.good}
  \begin{enumerate}[(a)]
  \item Singletons are good;
  \item A subset of a good subset is good;
  \item A (possibly infinite) union of good subsets with a non trivial
    intersection is good.
  \end{enumerate}
\end{axioms}
\begin{lemma}
  \label{lem:trivia}
  A collection $\goodsubset$ of subsets of $E$ satisfies
  Axioms~\ref{axioms.good} if and only if every subset of a maximal
  member of $\goodsubset$ belongs to $\goodsubset$ and the maximal
  members of $\goodsubset$ form a set partition of $E$.
\end{lemma}
\begin{proof} 
  Suppose that $\goodsubset$ satisfy Axioms~\ref{axioms.good}
  above. The fact that subsets of maximal members of $\goodsubset$
  belong to $\goodsubset$ follows directly from (b). Now take some
  $a\in E$.  By (a) and (c), the union $\goodsubset(a)$ of all good subsets containing
  $a$ is good. Furthermore, any maximal good subset of $E$ (for
  inclusion) is of this form. Using (c), two maximal good subsets
  $\goodsubset(a)$ and $\goodsubset(b)$ either coincide or are
  disjoint. Hence, the collection $\minmorphdec(\goodsubset):=
  \{\goodsubset(a)\}_{a\in E}$ forms a set partition of $E$ into good
  subsets. The converse is immediate.
\end{proof}

Let $\goodsubset$ be a collection of subsets of $E$ satisfying
Axioms~\ref{axioms.good}. The \emph{components} of $\goodsubset$ are
the maximal members of $\goodsubset$.  We denote by $\mathcal P(\goodsubset)$ the partition
of $E$ into components of $\goodsubset$ and by $s(\goodsubset)$
the number of components,  that is the size of $\mathcal P(\goodsubset)$.
 
We consider the refinement order on the set $\setpartitions(E)$ of the
set partitions on $E$, choosing by convention that $\mathcal P\coarser
\mathcal Q$ if $\mathcal P$ is coarser than $\mathcal Q$ (that is each
block of $\mathcal P$ is a union of blocks of $\mathcal Q$). With this
ordering $\setpartitions(E)$ forms a lattice, with $\{E\}$ as minimal
element and $\{\{a\}\suchthat a\in E\}$ as maximal element.

A set partition $\mathcal P:=(P_x)_{x\in X}$ of $E$ is \emph{good} if
each block $P_x$ is good. For example, $\mathcal P(\goodsubset)$ is
good and so is the trivial partition into singletons. From Lemma~\ref{lem:trivia} above we have:
\begin{proposition}
  \label{proposition.good_set_partitions}
  A set partition $\mathcal P$ is good if and only if it refines
  $\minmorphdec(\goodsubset)$. Hence the set of all good set partitions is a
  principal filter of the lattice $\setpartitions(E)$; in particular,
  it is stable under joins and meets.
\end{proposition}

Due to this fact, we name $\minmorphdec(\goodsubset)$ the
\emph{minimal good set partition}.

Take now a set $E$, and associate to each subset $D$ of $E$ a
collection $\goodsubset_D$ of subsets of $D$, called \emph{$D$-good},
and satisfying the axioms above. Assume further that the following
axiom is satisfied:
\begin{axiom}[Goodness axioms, continued]
  \label{axioms.goodd}
  \begin{enumerate}[(d)]
  \item If $D_1\subseteq D_2\subseteq E$ and $F\subseteq D_2$ is $D_2$-good,
    then $F\cap D_1$ is $D_1$-good.
  \end{enumerate}
\end{axiom}

Take $D_1\subseteq D_2$. A set partition $\mathcal P$ of $D_2$ induces
(by intersection of each of its block with $D_1$ and removal of the
empty ones) a set partition of $D_1$. The latter is \emph{properly
  induced} by $\mathcal P$ if $D_1$ intersects non trivially each
block of $\mathcal P$ (that is no block vanishes in the process).

Axiom (d) implies right away the following properties.
\begin{proposition}
  \label{proposition.good}
  Let  $D_1\subseteq D_2 \subseteq D_3\subseteq E$.

  \begin{enumerate}[(a)]
  \item Any $D_2$-good set partition of $D_2$ induces a $D_1$-good set
    partition of $D_1$.
  \item The size of the minimal good set partition of $D_2$ is at least that
    of the minimal good set partition of $D_1$. If it is finite there is equality if and only if 
    the minimal good set partition $\minmorphdec(D_1)$ of $D_1$ is
    properly induced by $\minmorphdec(D_2)$.
  \item \label{remarks.good.intermediate} If $\minmorphdec(D_1)$ is
    properly induced by $\minmorphdec(D_3)$, then $\minmorphdec(D_1)$ is
    properly induced by $\minmorphdec(D_2)$ which itself is properly
    induced by $\minmorphdec(D_3)$
  \end{enumerate}
\end{proposition}

To derive the main desired property, we need to add a last axiom on
good subsets.
\begin{axiom}[Goodness axioms, continued]
  \label{axioms.goode}
  \begin{enumerate}[(e)]
  \item If $\mathcal D$ is a chain of subsets of $E$ whose union is $D$, then a
    subset $F$ of $D$ is $D$-good as soon as $F\cap D'$ is $D'$-good
    for all $D'\in \mathcal D$.
  \end{enumerate}
\end{axiom}

\begin{lemma}
  \label{lemma.cardinal}
  Every subset $D$ of $E$ such that $s(\goodsubset_{D})$ is finite includes some finite subset $F$ such that
  $s(\goodsubset_{F})=s(\goodsubset_{D})$.
\end{lemma}
\begin{proof}
  We argue by induction on the cardinality $\kappa$ of $D$. If
  $\kappa$ is finite there is nothing to prove. Suppose otherwise that
  $\kappa$ is infinite, and write $D$ as the union of a chain
  $\mathcal D$ of subsets $D'$ of $D$, such that $|D'|<\kappa$.  By
  induction, we may assume that the property holds for each $D'$ in
  $\mathcal D$. There remains to prove that $s(\goodsubset_{D'})
  =s(\goodsubset_{D})$ for some $D'\in\mathcal D$. According
  to Proposition~\ref{proposition.good} (a) we have
  $s(\goodsubset_{D'})\leq s(\goodsubset_{D''})\leq
  s(\goodsubset_D)$  for all $D'\subseteq D''\subseteq D$. Since $s(\goodsubset_D)$ is finite, there is some
  $k$ and some $D'\in \mathcal D$ such that $s(\goodsubset_{D''})= k$
  for all $D''\in\mathcal D$ such that $D'\subset D''$. Set
  $X:=\{1,\dots,k\}$. For each such $D''$, denote $(D''_{x})_{x\in X}$
  the minimal good set partition $\minmorphdec(\goodsubset_{D''})$ of
  $D''$.  Without loss of generality, and up to some renumbering of
  the blocks, we may assume that, for each $x\in X$, the set $\mathcal
  D_x:=\{ D''_x \suchthat D' \subseteq D''\in \mathcal D \}$ forms a chain for
  inclusion. Using Axiom (e), the union $D_x$ of those is a $D$-good
  block. Then, $(D_x)_{x\in X}$ is a good set partition of $D$ into
  $k$ blocks and thus it is the minimal one, proving that
  $s(\goodsubset_{D'})=k=s(\goodsubset_{D})$, as desired.
\end{proof}

\begin{proposition}
  \label{proposition.chain_good_set_partitions}
  Assume that $E$ has a good partition into finitely many blocks, and
  let $\mathcal F$ be the collection of all subsets $D$ of $E$ whose
  minimal good set partition is properly induced by that of $E$. Then,
  a subset $D$ is in $\mathcal F$ if and only if it includes a finite
  subset in $\mathcal F$.
\end{proposition}
\begin{proof}
  The ``if'' part follows from Proposition~\ref{proposition.good}. The
  converse follows from Lemma~\ref{lemma.cardinal}.
\end{proof}

\subsection{Minimal monomorphic decompositions}

We start with some elementary properties of the monomorphic
decompositions of a relational structure $R$ with base set $E$, and in
particular we show that their blocks satisfy the goodness axioms.

\begin{lemma}
  \label{remark.definition.monomorphic_set}
  Let $R$ be a relational structure with base set $E$. For a subset $F$ of $E$, the following conditions are equivalent:
  \begin{enumerate}[(i)]
  \item there exists some monomorphic decomposition of $R$ admitting $F$ as
   a  block;
  \item the set partition $\{F\} \cup \{ \{x\} \}_{x\not\in F}$ is a
    monomorphic decomposition of $R$;
  \item for every pair $A, A'$ of subsets of $E$  with the same finite cardinality the induced structures on $A$ and $A'$ are isomorphic whenever
    $A\setminus F=A'\setminus F$.
  \end{enumerate}
\end{lemma}
A subset $F$ is  a \emph{monomorphic
 part of $R$} if any (and therefore all) of the previous conditions hold.

Note that in $(iii)$ above, we do not not impose that the induced
substructures on $A$ and $A'$ are isomorphic via an isomorphism which
is the identity on $A\setminus F$. However, this condition is
fulfilled as soon as $F$ is an infinite monomorphic part which is
maximal w.r.t. inclusion (see Theorem~\ref{monomorphiccomp}; beware of
a different use of the letter $F$ in this result).
\begin{lemma}
  \label{lemma.set_partition_monomorphic_blocks}
  A set partition $(E_x)_{x\in X}$ of $E$ is a monomorphic
  decomposition of $R$ if and only if it is made of monomorphic parts of $R$.
\end{lemma}
\begin{proof}
  The ``only if'' part is by definition. Consider now a set partition
  whose blocks are monomorphic parts of $R$, and take $A$ and $A'$ such that
  $X^A=X^{A'}$. Let  $E_1,\dots,E_k$ be the blocks that $A$ and $A'$
  intersect non trivially, and for $l\in 0,\dots,k$, set $A_0:= A'$,  $A_l: =
  (A\cap E_1) \cup \dots \cup (A\cap E_l) \cup(A'\cap E_{l+1})\cup\dots\cup
  (A'\cap E_k)$. Then,  for $l<k$, $A_l$ and $A_{l+1}$ have the same
  cardinality and $A_l\backslash E_{l+1}=A_{l+1}\backslash E_{l+1}$ are
  equal. Since $A_0=A'$ and $A_k=A$, one has $A\approx A'$, as desired.
\end{proof}

\begin{lemma} 
  \label{lemma.monomorphic_parts_are_good}
  Let $R$ be a relational structure on a set $E$. Then the set of 
  monomorphic parts of $R$ satisfies 
  Axioms~\ref{axioms.good}. If furthermore for every subset $D$ of $E$ the good sets of $D$ consist of the monomorphic parts of $R_{\restriction  D}$ then Axioms~\ref{axioms.goodd}, and~\ref{axioms.goode} are satisfied.
\end{lemma}
\begin{proof}
 We denote by $A$ and $A'$  two subsets of $E$ of the same finite
  cardinality.

  Axiom~\ref{axioms.good}(a): if $F$ is a singleton, and $A\setminus F=A'\setminus F$,
  then $A$ and $A'$ are equal, hence trivially isomorphic.

  Axiom~\ref{axioms.good}(b): let $F$ be a monomorphic part of $R$ and $F'\subseteq F$. If
  $A\setminus F'=A'\setminus F'$, then $A\setminus F=A'\setminus F$,
  and therefore $A$ and $A'$ are isomorphic.

  Axiom~\ref{axioms.good}(c): let $(F_i)_{i\in I}$ be a (possibly infinite) family of
 monomorphic parts which all share at least a common point $x$, and set
  $F=\bigcup_{i\in I} F_i$, and assume $A\setminus F=A'\setminus F$.

  Without loss of generality, we may assume that $A$ and $A'$ differ
  by a single point: $A=B\cup \{a\}$ and $A'=B\cup \{a'\}$ with $a$
  and $a'$ in $F$ (otherwise, pickup a sequence $A:=A_1,\dots,A_l=A'$,
  where at each step $A_j\setminus F = A\setminus F$, and $A_j$ and
  $A_{j+1}$ differ by a single point; then use transitivity).

  If $\{a,a'\}$ is a subset of some $F_i$, then it forms a monomorphic
 part. Using that $A\setminus \{a,a'\}=B=A'\setminus \{a,a'\}$, wet
  get  $A\approx A'$.

  Otherwise, take $F_i$ and $F_{i'}$ such that $a\in F_i$ and $a'\in
  F_i'$.

  If $x\not\in B$, then using successively that $\{a,x\}\subseteq F_i$
  is a monomorphic part  and $\{a',x\}\subseteq F'_i$ is a monomorphic
 part,  one gets,
  \begin{equation}
    A = B\cup \{a\} \approx B\cup\{x\} \approx B\setminus \cup \{a'\} = A'\,.
  \end{equation}
  Similarly, if $x\in B$,
  \begin{equation}
    A = B\cup \{a\} \approx B\setminus\{x\} \cup \{a,a'\} \approx
    B\setminus \cup \{a'\} = A'\,.
  \end{equation}

  Axiom~\ref{axioms.goodd}(d): Let $F$ be $D_2$-good, and assume that $A$ and $A'$ are
  subsets of $D_1$ such that $A\backslash (F\cap D_1)=A'\backslash
  (F\cap D_1)$. Then, $A\backslash F=A\backslash (F\cap
  D_1)=A'\backslash (F\cap D_1)=A'\backslash F$; hence $A\approx A'$.

  Axiom~\ref{axioms.goode}(e): assume $A$ and $A'$ are subsets of $D$ such that
  $A\backslash F=A\backslash F$. Take $D'$ large enough in $\mathcal
  D$ so that $D'$ contains both $A$ and $A'$. Then $A\backslash (F\cap
  D')=A\backslash F = A'\backslash F = A'\backslash (F\cap D')$, and
  therefore $A\approx A'$.
\end{proof}

We can now specialize
Proposition~\ref{proposition.good_set_partitions} to monomorphic
decompositions, to prove and refine Proposition~\ref{proposition.monodec}.
\begin{proposition}
  \label{proposition.lattice_monomorphic_decomposition} %
  The maximal monomorphic parts of $R$ form a monomorphic
  decomposition of $R$. Furthermore, the other monomorphic
  decompositions are exactly the finer set partitions of this
  partition.
\end{proposition}

The \emph{monomorphic components} of $R$, or \emph{components} for
short, are the maximal mono\-morphic parts of $R$. We denote by $c(R)$
the number of components of $R$. The partition of $E$ into components,
that we denote by $\mathcal P(R)$, is the \emph{minimal monomorphic
  decomposition of} $R$.

The main consequence of Proposition~\ref{proposition.lattice_monomorphic_decomposition} is the following:
\begin{corollary} \label{cor:shape} Let $R$ and $R'$ be two relational
  structures with domains $E$ and $E'$ respectively. If $R$ and $R'$
  are isomorphic, then the components of $R$ are mapped bijectively
  onto the components of $R'$ by any isomorphism $\sigma$. In
  particular, every automorphism of $R$ induces a permutation of the
  components of $R$ and, if the domain $E$ of $R$ is finite, $E$ and
  $E'$ have the same number of components and the same shape
  w.r.t. their minimal decompositions.
 \end{corollary}

Propositions~\ref{proposition.good}
and~\ref{proposition.chain_good_set_partitions} also apply, and will
be used in the sequel. In particular, we get the following:
\begin{proposition}
  \label{proposition.restriction_monomorphic_decomposition}
  Let $R$ be a relational structure and $R'$ be a restriction of
  it. Then any monomorphic decomposition of $R$ induces a monomorphic
  decomposition of $R'$. Note however that this monomorphic
  decomposition may have fewer components, and that minimality is not
  necessarily preserved.
\end{proposition}

Let $(E_x)_{x\in X}$ be a partition of a set $E$. Given $d\in \N$,
call \emph{$d$-fat} a subset $A$ of $E$ such that, for all $x\in X$,
$d_x(A) \geq d$ whenever $A\not\supseteq E_x$. We prove that, if the minimal
monomorphic decomposition has finitely many components, then the
isomorphism relation is shape preserving on fat enough sets, and we
derive a lower bound on the profile.

\begin{lemma}
  \label{lemma.minimalGrowthRate}
  Let $R$ be an infinite relational structure on a set $E$ admitting a finite  monomorphic decomposition. Then, there exists some
  integer $d$ such that on every $d$-fat subset $A$ of $E$ (w.r.t.   $\mathcal P(R)$)  the partition $\mathcal P(R)_{\restriction A}$  induced by $\mathcal P(R)$ on $A$ coincides  with $\mathcal P(R_{\restriction A})$. In particular the shape of $A$ w.r.t. $\mathcal{P}(R)$ coincides 
  with the shape of $A$ w.r.t. $\mathcal {P}(R_{\restriction A})$. 
   \end{lemma}

\begin{proof}
Let $E$ be the base set of $R$ and  $(E_x)_{x\in X}$ be the minimal partition of $E$ into monomorphic parts. 
According to Proposition~\ref{proposition.chain_good_set_partitions} there is a family $\mathcal F$ of finite subsets $F$ of $E$ such that for every subset $D$ of $E$ the minimal partition of $D$ into monomorphic parts of $R_{\restriction D}$ is properly induced by the partition of $E$ if and only if $D$ contains some member of $\mathcal F$. Pick $F\in \mathcal F$. We claim that  every subset  $F'$ of $E$ such that  $d(F')=d(F)$ has the same property as $F$. Indeed, according to  (c) of Proposition~\ref{proposition.good},  it suffices to prove that 
$c(R_{\restriction F'})=c(R)$.  This fact is straightforward, indeed, since $d(F')=d(F)$, $R_{\restriction F}$ is isomorphic to $R_{\restriction F'}$,  hence from Corollary~\ref{cor:shape}  the  components   of $R_{\restriction F}$ correspond bijectively to the  components of  $R_{\restriction F'}$. Hence, $c(R_{\restriction F})=c(R_{\restriction F'})$.  By our choice of $F$, $c(R_{\restriction F})=c(R)$. Since the components of $R$ induce a monomorphic decomposition of $R_{\restriction F'}$, $c(R_{\restriction F'})\leq c(R)$. This yields $c(R_{\restriction F'})=c(R)$, as required. Let $d:= \max \{\vert F\cap E_x\vert \suchthat x\in X\}$. Let $A$ be a $d$-fat subset of $E$ w.r.t. the partition $(E_x)_{x\in X}$. Then $A$ contains a subset $F'$ such that $d(F')=d(F)$ and we are done.
\end{proof}

With the notations of Lemma~\ref{lemma.minimalGrowthRate}
we have \begin{corollary} Two isomorphic $d$-fat subsets of $E$ always have the
  same shape w.r.t. $\minmorphdec(R)$.
\end{corollary}

\begin{theorem}
  \label{theorem.minimalGrowthRate} %
  Let $R$ be an infinite relational structure with a finite
  monomorphic decomposition, and let $k$ be its monomorphic dimension.
  Then, the profile $\profile_R$ of $R$ is bounded from below by a
  polynomial of degree $k-1$; namely, there exists $n_0$ such that for
  $n\geq n_0$, $\profile_R(n)\geq \wp_k(n-n_0)$, where $\wp_k(m)$ is the
  number of integer partitions of $m$ in at most $k$ parts.
\end{theorem}
\begin{proof}
  Take $d$ as in Lemma~\ref{lemma.minimalGrowthRate}, and let
  $n_0:=k_1d+m$, where $k_1$ is the number of monomorphic components
  of $R$ having at least $d$ elements, and $m$ is the sum of the
  cardinalities of the other finite monomorphic components of $R$.

  Let $E_1, \dots, E_p$ be the monomorphic
  components of $R$, enumerated in such a way that $\vert E_1\vert
  \geq \cdots \geq \vert E_p\vert$. To each decreasing sequence $x:=
  x_1\geq \cdots \geq x_{k'}$ of positive integers such that $k'\leq
  k$ and $x_1+\cdots+ x_{k'}= n-n_0$ associate an $n$-element subset
  $A(x)$ of $E$ such that $\vert A(x)\cap E_i\vert$ is respectively
  $d+x_i$ if $i\leq k'$, $d$ if $k'<i\leq k_1$ and $\vert E_i\vert $
  if $k_1<i\leq p$. The set $A(x)$ is $d$-fat and its shape is $(\vert
  A(x)\cap E_1\vert, \dots, \vert A(x)\cap E_p\vert)$. Clearly, if $x$
  and $x'$ are two distinct sequences as above, the shapes of $A(x)$
  and $A(x')$ are distincts. Since these sets are $d$-fats, the
  restrictions $R_{\restriction A(x)}$ and $R_{\restriction A(x')}$
  are not isomorphic. The claimed inequality follows.

  As it is well known, $\wp_k(m)$ is asymptotically equivalent to
  $\frac{m^{k-1}}{(k-1)!k!}$ (see
  e.g.~\cite{VanLint_Wilson.ACourseInCombinatorics.1992}). Thus
  $\profile_R(n)$ is asymptotically bounded below by
  $\frac{n^{k-1}}{(k-1)!k!}$.
\end{proof}

\subsection{Monomorphic decompositions, chainability and a proof of Theorem~\ref{thm:charcfinitemonomorph}}
\label{thm:charcfinitemonomorph.proof}

In Theorem~\ref{monomorphiccomp} below, we present the relationship between the notions of chainability, monomorphy and  our notion of monomorphic parts. The key ingredients of the proof are Ramsey's theorem, Compactness theorem of first order logic and some properties of the kernel of relational structures. We give  below the facts we need.

Let $R:= (E, (\rho_i)_{i \in I})$ a relational structure. For each subset $I'$ of $I$ we set $R^{I'}:= (E, (\rho_i)_{i \in I'})$. 

Compactness theorem of first order logic yields the following lemma: 

\begin{lemma}\label{lemma:compactness}A relational structure $R:= (E, (\rho_i)_{i \in I})$ is $F$-chainable if and only if for each finite subset $F'$ of $F$ and every finite subset $I'$ of $I$, $R^{I'}_{\restriction (E\setminus F)\cup F'}$ is $F'$-chainable. 
\end{lemma}

Ramsey's theorem  yields the following result:
\begin{lemma}[Fra\"{\i}ss\'e~\cite{Fraisse.1954}]
  \label{chainablerestriction}
  Let $R$ be a relational structure with domain $E$, and $F$ be a
  finite subset of $E$. If the signature of $R$ is finite then there
  is an infinite subset $E'$ of $E$ containing $F$ on which the
  restriction $R':= R_{\restriction E'}$ is $F$-chainable.
\end{lemma}

We also need some properties of the kernel. Most of these properties
are based on the following simple lemma (see~\cite{Pouzet.1979.RM} for
finite signature and~\cite[3 of Lemma 2.12]{Pouzet_Sobrani.2001} for
the general case).
\begin{lemma} \label{minusab}
  For all $a,b\in E$,\,
$\mathcal {A}(R_{\restriction E\setminus \{a\}})=\age (R)\,\Longrightarrow\, \mathcal {A}(R_{\restriction E\setminus \{a, b\}})=\mathcal {A}(R_{\restriction E\setminus \{b\}})$.
\end{lemma}

From this property, we  have easily:
\begin{lemma}\label{lem:restrictkernel}
Let $R$ be a relational structure with domain $E$, and $E'$ be a subset of $E$. If $K(R)\subseteq E'$ and $E\setminus E'$ is finite then $R$ and $R_{\restriction E'}$ have the same age and the same kernel.
\end{lemma}

\begin{lemma}\label{lem:restrictkernel2}
  Let $R$ be a relational structure with domain $E$, and $F$ be a
  finite subset of $E\setminus K(R)$. Assume furthermore that $K(R)$
  is finite. Then, for every finite subset
  $A$ of $E$ there is some subset $A'$ of $E\setminus F$ such that
  $A\approx A'$ and $A\cap K(R)=A'\cap K(R)$.
\end{lemma}
The proof of this lemma can obtained via the existence of a finite
subset of $E$ localizing $A\cap K(R)$ (see Proposition 2.17
of~\cite{Pouzet_Sobrani.2001}).

Trivially, we have:
\begin{lemma} \label{lem:kernelmonopart}If $E'$ is an infinite monomorphic part of $R$  then $K(R)\cap E'=\emptyset$. 
\end{lemma}

Using Lemma~\ref{lemma:compactness} and  Lemma~\ref{lem:restrictkernel2} we get:
\begin{lemma}\label{lem:kernelchainable} If $R$ is almost-chainable  then $R$ is $K(R)$-chainable. 
\end{lemma}

 \begin{theorem} \label{monomorphiccomp}
Let $R:= (E, (\rho_i)_{i \in I})$ be a relational structure, $E'$ be a
subset of $E$ and $F:= E'\setminus E$. Let us consider the following
properties:
 \begin{enumerate}
 \item[{(i)}]   $R$ is $F$-chainable;
\item[(ii)] $R$ is $F$-monomorphic;
\item[(iii)] $E'$ is a monomorphic part of $R$.
\end{enumerate}
Then $(i)\Rightarrow (ii)\Rightarrow (iii)$. If $E'$ is infinite then
$(ii)\Rightarrow (i)$. If $E'$ is infinite and $E'$ is a monomorphic
component of $R$ then $(iii)\Rightarrow (i)$.
 \end{theorem}
 
 \begin{proof}
 The  implications $(i)\Rightarrow (ii)\Rightarrow (iii)$ are obvious. Implication  $(ii) \Rightarrow (i)$ when  $E'$ is infinite is due to  Fra\"{\i}ss\'e. A proof in the case of a finite signature is in \cite{Fraisse_Pouzet.1971}. We just recall the principle of the proof. We prove first that  for each finite subset $I'$ of $I$,  each finite subset $F'$ of $F$ and each finite subset $E''$ of $E'$, the restriction  $R':= R^{I'}_{E'' \cup F'}$ is $F'$-chainable. For that we apply   Lemma
\ref{chainablerestriction}.  It yields  an infinite subset $E'_1$ of $E'$ on which $R'_1= R^{I'}_{E'_1 \cup F'}$ is $F$-chainable. Since  $R$ is $F$-monomorphic, there is a local  isomorphism fixing $F'$ pointwise which send  $E'' \cup F'$ into $E'_1\cup F'$; thus $R':= R^{I'}_{\restriction E'' \cup F'}$ is $F'$-chainable. We conclude with  Lemma~\ref{lemma:compactness}.  The proof of  $(iii)\Rightarrow (i)$ when $E'$ is a component relies on the following claim: 

{\bf Claim 1.} {\it Let $F'$ be a finite subset of $F$. There is some finite subset $F''$ of $F$ containing $F'$ such that $K(R_{\restriction E'\cup F''})=F''$.}

{\it Proof of Claim~1.}
Let $x\in F$. Since the components of $R$ are the maximal monomorphic
parts of $R$, the set $E'\cup \{x\}$ is not a monomorphic part of
$R$. Hence, there are $n_x$ and $A_{x}\not\approx A'_{x}$ in
$[E]^{n_x}$ such that
\begin{equation}
  \label{equ:claimmocomp1}
  A_{x}\setminus (E'\cup \{x\})= A'_{x}\setminus (E'\cup \{x\})\,.
\end{equation}
For each $x\in F'$ select $A_x$ and $A'_x$ as above.  Define
$V:=\bigcup_{x\in F'} A_x \cup A'_x$, as well as $R':= R_{\restriction
  E'\cup V}$, and $F'':=K(R')$.

{\bf Subclaim 1.} $K(R_{\restriction E'\cup F''})=F''$.

{\it Proof of Subclaim 1.} Since $(E'\cup V)\setminus (E'\cup F'')$ is finite and $F''=K(R')$, Lemma~\ref{lem:restrictkernel} asserts that $R_{\restriction E'\cup V}$ and $R_{\restriction E'\cup F''}$ have the same age and the same kernel. \hfill \qed

{\bf Subclaim 2.} $F'\subseteq F''\subseteq F$. 

{\it Proof of Subclaim 2.} Since $E'$
is a monomorphic component of $R$, this is a monomorphic part of $R'$ and Lemma~\ref{lem:kernelmonopart} asserts that $K(R')$ is disjoint from $E'$. This
yields $F''\subseteq V\setminus E'\subseteq F$. Suppose that $F'\not \subseteq F''$. Let $x\in F'\setminus F''$. Then, since  $F''\subseteq V\setminus E'$, we have $F'' \subseteq E\setminus (E'\cup \{x\})$. Thus 
from~\eqref{equ:claimmocomp1} we have $A_{x}\cap F''= A'_{x}\cap F''$. Since $V\setminus (E'\cup F'')$ is finite, it follows from Lemma~\ref{lem:restrictkernel2} that  there are $A, A'\in 
[F''\cup E']^{n_{x}}$ such that $A\cap F''=A_{x}\cap F''$, $A'_{x}\cap F''=
A' \cap F''$, and  the restrictions $R'_{\restriction A}$ and
$R'_{\restriction A'}$ are respectively isomorphic  to
$R'_{\restriction A_{x}}$ and $R'_{\restriction A'_{x}}$. Since $E'$
is a monomorphic component of $R'$,  $R'_{\restriction A}$ and $R'_{\restriction A'}$ are
isomorphic. However this implies that $A_{x}\approx A'_{x}$, a
contradiction. Therefore, $F'\subseteq F''\subseteq V$. This completes the proof of the claim.  \hfill \qed

With these two subclaims, the proof of Claim 1 is complete. \hfill \qed

Now, since $E'$ is a monomorphic part of $R_{\restriction E'\cup F''}$, the profile of $R_{\restriction E'\cup F''}$ is bounded. According to the implication $(a)\Rightarrow (c)$ in Theorem~\ref{theorem.boundedProfile}, $R_{\restriction E'\cup F''}$ is almost chainable. According to Lemma~\ref{lem:kernelchainable}, $R_{\restriction E'\cup F''}$ is $K( R_{\restriction E'\cup F''})$-chainable, that is $F''$-chainable. It follows that $R_{\restriction  E'\cup F'}$ is $F'$-chainable. Since this holds for every finite subset $F'$ of $F$, it follows from Lemma~\ref{lemma:compactness} that $R$ is $F$-chainable.  Hence $(ii)$ holds. This completes the proof of Theorem~\ref{monomorphiccomp}.\end{proof}

Theorem~\ref{thm:charcfinitemonomorph} follows immediately from
Theorem~\ref{monomorphiccomp}. Indeed, if $R$ admits a finite
monomorphic decomposition, we may choose one such that all the finite
blocks are singletons. Hence, if $(E_x)_{x\in X}$ is such a
decomposition, then a lexicographical sum, in any order, of the chains
$(E_x, \leq_x)$ given by (i) of Theorem~\ref{monomorphiccomp} yields a
linear order on $E$ for which the $E_x$ are intervals and every local
isomorphism preserving the $E_x$'s preserves $R$. %

In the special case of relational structures made of unary or binary relations, Theorem~\ref{thm:charcfinitemonomorph} yields the following characterization:

\begin{corollary} \label{cor:binary}If a relational structure $R$ is at most binary, it has a finite monomorphic decomposition if and only if  it is a lexicographical sum of chainable relational structures indexed by  a finite relational structure.
\end{corollary}

Let us say that a relational structure $R:= (E,
(\rho_i)_{i \in I})$ is  \emph{almost multichainable} if there is a finite subset $F$ of $E$ and an enumeration
$(a_{x, y})_{(x,y)\in V\times L}$ of the elements of $E\setminus F$ by a
set $V\times L$, where $V$ is finite and $L$ is a linearly ordered set 
such that for every local isomorphism $f$ of $L$, the map $(1_V, f)$
extended by the identity on $F$ is a local isomorphism of $R$ (the map
$(1_V, f)$ is defined by $(1_V,f)(x, y):= (x, f(y))$). If the map $(1_V, f)$
extended by the identity on $F$ is a local isomorphism of $R$ for every permutation $f$ of $L$, $R$ is \emph{cellular}.

The notion of almost multichainability was introduced in
\cite{Pouzet.TR.1978} and appeared in \cite{Pouzet.1981.RI, Pouzet.1979.RM}. It was shown that a relational structure  with polynomially bounded profile and finite kernel has the same age as an almost multichainable relational structure \cite{Pouzet.2006.SurveyProfile}. Cellularity was introduced by J. Schmerl \cite{schmerl.1990} in 1990.   In  \cite{Pouzet.2006.SurveyProfile} it is shown that a graph has a polynomially bounded profile if and only if it is cellular (see  \cite{Pouzet.2006.SurveyProfile}). 

The following easy corollary of Theorem~\ref{thm:charcfinitemonomorph} show that relational structures with a finite monomorphic decomposition are essentially almost multichainable (the converse is far from be true).

\begin{corollary}\label {lem2}
If an infinite relational structure has a finite monomorphic decomposition, then it has some restriction having the same age which is almost multichainable. \end{corollary}

\section{Proof of Theorem~\ref{theorem.quasipolynomial}}
\label{section.quasipolynomial}

\subsection{Preliminary steps}

From now on, we assume that $R$ has a finite monomorphic decomposition
$(E_x)_{x\in X}$. The growth rate of the profile $\profile_R$ is at
most $n^{k-1}$ where $k=\vert X_{\infty}\vert $; indeed
$\profile_R(n)$ is bounded above by the number of integer vectors
$(d_x)_{x\in X}$ such that $d_x\leq |E_x|$ and $\sum_{x\in X}d_x=n$; a
more algebraic explanation is that the association $A\mapsto X^A$
makes the age algebra into (essentially) a subalgebra of $K[X_\infty]$
(see~\cite{Pouzet_Thiery.AgeAlgebra2}). If the decomposition is
minimal then, according to Lemma~\ref {lemma.minimalGrowthRate}, the
growth rate of the profile is at least $n^{k-1}$. We now turn to the
proof that the Hilbert series is a rational fraction.

We define a total order on monomials in $\K[X]$ by comparing their
shapes w.r.t. the degree reverse lexicographic order and breaking ties
by the usual lexicographic order on monomials for some arbitrary fixed
order on $X$ (recall that the \emph{reverse lexicographic order}
$\leq_{\revlex}$ is defined as follow: for two integer sequences
$d:=(d_1,\dots,d_n)$ and $d':=(d'_1,\dots,d'_n)$ of same length,
$d<_{\revlex} d'$ if $d\ne d'$ and $d_i>d'_i$ where $i$ is the largest
$j\leq n$ such that $d_j\not = d'_j$; for example $(4, 1, 1)
<_{\revlex}(3,3,0)$; for the \emph{degree reverse lexicographic
  order}, one first compares the sum of the two sequences and then
break ties with reverse lexicographic order; for example
$(3,2,0)<_{\degrevlex}(4,1,1)$). Beware that this total order on
monomials is well founded but not a \emph{monomial order}.
We define the \emph{leading monomial} $\lm(\tau)$ of an
isomorphism type $\tau$ as the unique maximal monomial in the set
$\{X^{A} \suchthat \type A=\tau\}$.

To prove the theorem, we essentially endow the set of leading
monomials with a monomial ideal structure in some appropriate
polynomial ring.  The point is that the Hilbert series of such
monomial ideals are simple rational fractions (see e.g.~\cite[Chapter
9, \S 2]{Cox_al.IVA}; note that their presentation is in term of the
\emph{Hilbert function}, but this is equivalent).
\begin{proposition}
  \label{proposition.HilbertSeriesMonomialIdeal}
  Let $\K[x_1,\dots,x_n]$ be a polynomial ring whose variables have
  positive degree $d_i:=\deg(x_i)$, and let $\ideal$ be a monomial ideal.
  Then, the Hilbert series of $\ideal$ is of the form:
  \begin{displaymath}
    \frac{P(Z)}{(1-Z^{d_1})\cdots(1-Z^{d_n})}\,.
  \end{displaymath}
  where $P\in \Z[Z]$.
\end{proposition}
We include the proof, as it is short and sheds some light for our
purpose.
\begin{proof}
  First, the Hilbert series of a principal ideal $\K[x_1,\dots,x_n].m$
  generated by a monomial $m$ of degree $d$ is
  \begin{displaymath}
    \frac{Z^d}{(1-Z^{d_1})\cdots(1-Z^{d_n})}\,.
  \end{displaymath}
  Furthermore, the intersection of two principal ideals is again
  principal. Take now any monomial ideal $\ideal$. By Dickson Lemma, it is finitely
  generated by monomials $m_1,\dots,m_r$. Therefore, the Hilbert series
  of $\ideal$ can be computed by inclusion-exclusion from that of the
  principal ideals
  $(\K[x_1,\dots,x_n].m_{i_1}\cap\dots\cap\K[x_1,\dots,x_n].m_{i_s})_{1\leq i_1<\dots<i_s\leq r}$.
\end{proof}

The key property of leading monomials of age algebras is reminiscent
of Stanley-Reisner rings.  To each set $S\subseteq X$, associate the
monomial $x_S := \prod_{i\in S} x_i$. By square free factorization,
any monomial $m\in\K[X]$ can be written in a unique way as a product
$m = x_{S_1}^{e_1} \dots x_{S_r}^{e_r}$ where $\emptyset \subset
S_1\subset\dots\subset S_r \subset X$ is a chain of non empty
subsets of $X$, and the $e_i$ are positive. Each $S_i$ is a
\emph{layer} of $m$, and $S_1\subset\dots\subset S_r$ is the \emph{chain support} of $m$.

\begin{lemma}\label{lemma.addlayer}
  Let $m$ be a leading monomial, and $S\subseteq X$ be a layer of $m$.
  Then, $m x_S$ is again a leading monomial unless $d_i=|E_i|$ for
  some $i$ in $S$.
\end{lemma}
\begin{proof}
  \let\isomorphism=\psi
  Suppose that $d_i<|E_i|$ for every $i$ in $S$. Let $A,A',B,B'$ be subsets of $E$ such that $X^A = m$, $X^{B} = mx_{S}$,
  $X^{B'}$ is the leading monomial $\lm(\overline B)$, and $A'$ is any
  subset of $B'$ belonging to $\overline A$. Let $R_{\restriction A}$, $R_{\restriction B}$, $R_{\restriction A'}$ and
  $R_{\restriction B'}$ be the corresponding induced structures, and let  $\isomorphism$ be an isomorphism from $R_{\restriction B}$ to $R_{\restriction B'}$.

  Setting $e:= |X|$, write respectively $\overline d(A)=(\alpha_1,\dots,\alpha_e)$
  and $\overline d(B)=(\beta_1,\dots,\beta_e)$ the shapes of $A$ and $B$, and
  similarly for $A'$ and $B'$. Our first goal is to prove
  that the   shapes of $B$ and $B'$ are the same.

  {\bf Claim 1.} {\it $\alpha_p=\beta_{p}=\beta'_{p}=\alpha'_p$ for all $p> s$ where $s:=\vert S\vert$}.

 {\it Proof of Claim 1.}
   Using that $A$ and $A'$ have the same degree, and similarly for $B$
  and $B'$, we have:
  \begin{equation}
    (\alpha_1,\dots,\alpha_e) \geq_{\revlex} (\alpha'_1,\dots,\alpha'_e)
    \quad \text{and} \quad
    (\beta'_1,\dots,\beta'_e) \geq_{\revlex} (\beta_1,\dots,\beta_e)\,.
  \end{equation}
  Using that $X^B=X^A x_{S}$, we get:
  \begin{equation}
    (\beta_1,\dots,\beta_e) = (\alpha_1+1,\dots,\alpha_s+1,\alpha_{s+1},\dots,\alpha_e)\,.
  \end{equation}
  From the inclusion $A'\subset B'$, we deduce that:
  \begin{equation}
    (\alpha'_1,\dots,\alpha'_e) \geq_{\revlex} (\beta'_1,\dots,\beta'_e).
  \end{equation}
  Altogether, using that $\leq_\revlex$ is preserved on suffixes, we
  conclude that:
  \begin{equation}
    \begin{aligned}
      (\beta_{s+1},\dots,\beta_e)
      &= (\alpha_{s+1},\dots,\alpha_e)
      \geq_{\revlex} (\alpha'_{s+1},\dots,\alpha'_e)
      \geq_{\revlex} (\beta'_{s+1},\dots,\beta'_e)\\
      &\geq_{\revlex} (\beta_{s+1},\dots,\beta_e)\,,
    \end{aligned}
  \end{equation}
  and therefore all those suffixes coincide. \hfill \qed

  {\bf Claim 2.} {For each $i\in S$, $E_i\cap B$ is a component of
    $R_{\restriction B}$, that is a maximal monomorphic part of
    $R_{\restriction B}$.}

  {\it Proof of Claim 2.} Suppose not. Choose $i\in S$ such that
  $E_i\cap B$ is not a component of $R_{\restriction B}$
  and is of maximal cardinality with that property.
  Since $(E_k)_{k\in X}$
  is a monomorphic decomposition of $R$, then by
  Proposition~\ref{proposition.restriction_monomorphic_decomposition},
  $(E_k\cap B)_{k\in X}$ is a monomorphic decomposition of
  $R_{\restriction B}$. Hence, by
  Proposition~\ref{proposition.lattice_monomorphic_decomposition}, it
  refines the minimal monomorphic decomposition of $R_{\restriction
    B}$. Since $E_i\cap B$ is not a maximal monomorphic part, there is some index $j\in X$
  such that $(E_i\cup E_j)\cap B$ is still a monomorphic part of
  $R_{\restriction B}$. Due to our choice of $i$, we have $|E_i\cap
  B|\geq |E_j\cap B|$.  Note that this implies $|E_i\cap A|\geq
  |E_j\cap A|$. Let $a$ be the unique element of $E_i\cap (B\backslash
  A)$, pick $a'$ in $E_j\cap A$, and set
  $A'':=A\cup\{a\}\backslash\{a'\}$.  Since the elements $a$ and $a'$ belong to the same component of $R_{\restriction B}$, $A''\approx A$". We also
  have $\overline d(A)<_{\degrevlex} \overline d(A'')$, a contradiction with
  the fact that $X^A$ is a leading monomial.  \hfill \qed

  The proof of Claim~2 would have been simpler if the decomposition
  induced on $A$ by the $E_x$ were the components of $R_{\restriction
    A}$. However this is not true in general, even with the assumption
  that $X^A$ is a leading monomial. For a simple example, take for $R$
  the union of two non trivial cliques, and $A$ containing exactly one
  element in each clique; then $R_{\restriction A}$ has a single
  monomorphic component and not two.

  Let  $S':=\{i_1,\dots, i_s\}$ be $s$ distinct elements of $X$ such that $|E_{i_k}\cap B'|=\beta'_{k}$ for $k=1, \dots, s$. Set $U:=\bigcup_{i\not\in S} (E_i\cap B)$ and
  $U':=\bigcup_{i\not\in S'} (E_i\cap B')$.

  {\bf Claim 3.} {\it $U$ is the set of all $b\in B$ such that the
    $B\setminus \{b\}$ contains no member 
   of $\overline A$. The same statement holds for $U'$
    w.r.t. $B'$. In particular, $\isomorphism$ transforms $U$ into $U'$}.

  {\it Proof of Claim 3.} Let $b\in U'$. According to Claim 1,  the equality
  $(\alpha'_{s+1},\dots,\alpha'_e) = (\beta'_{s+1},\dots,\beta'_e)$
  holds for any $A'\subset B'$ such that $A\approx A'$. Hence,   no
  member of $\overline A$  is included in $B'\setminus \{b\}$. On the other hand, from
  the definition of $U$, $B\setminus \{b\}$ contains a member of 
  $\overline A$ for every element $b \in B\setminus U$. Since $|U|= \beta_{s+1}+\cdots+\beta_{e}=\beta'_{s+1}+\cdots+\beta'_{e}=|U'|$  and
  $\isomorphism$ is an isomorphism from $R_{\restriction B}$ to $R_{\restriction B'}$, the statement follows by
  cardinality count. \hfill \qed

  {\bf Claim 4.} {\it $\isomorphism$ transforms $(E_i\cap B, i\in S)$ into
    $(E_i\cap B', i\in S')$.}

  {\it Proof of Claim 4.} Let $P(B)$ be the minimal monomorphic
  decomposition of $B$. By Claim 2, it is of the form $\{E_i\cap B 
  \suchthat i\in S\}\cup \mathcal P$, where $\mathcal P$ is some partition of $U$.
  Then $\{\isomorphism(E_i\cap B)  \suchthat i\in S\}
  \cup \mathcal P'$,  where  $\mathcal P':=\{\isomorphism(C) \suchthat C\in \mathcal P\}$,  is the minimal monomorphic decomposition $P(B')$ of $B'$. Since   $(E_i\cap B')_{i\in X}$ is  a monomorphic
  decomposition of $B'$, this is a refinement of $P(B')$. Using
  that, from Claim 3,  $\isomorphism (B\setminus U)= B'\setminus U'$, we obtain that $(E_i\cap B')_{i\in S'}$ is a refinement of $(\isomorphism(E_i\cap B))_{ i\in S}$. Since these two decompositions have the same
  cardinality, they must coincide. The statement of Claim 4
  follows. \hfill \qed

  Claim 2 and Claim 4 imply immediately that $B$ and $B'$ have the
  same shape. Fix now $A':=\isomorphism(A)$. Using Claim 4, $A'$ is  obtained
  from $B'$ by removing exactly one element in each $E_{i}\cap B'$,
  $i\in S'$.
  Putting everything together, we have:
  \begin{itemize}
  \item $X^A$ and $X^{A'}$ have the same shape; similarly for $B$ and
    $B'$;
  \item $X^B=X^A x_S$, where $x_S$ is a layer of $X^A$; similarly
    $X^{B'}=X^{A'} x_{S'}$;
  \item $X^B\leq_\lex X^{B'}$ and $X^A\geq_\lex X^{A'}$.
  \end{itemize}
  Recall that, if two monomials have the same shape and at least one
  layer of size $s$, then lexicographic comparison is preserved upon
  changing the multiplicity of that layer (this is just applying a
  strictly monotone function to the exponents).
  Therefore, $X^A=X^{A'}$ and $X^B=X^{B'}$. In particular, $X^B=mx_S$
  is a leading monomial, as desired.
\end{proof}

\subsection{Final step}

  Fix a chain $C:=\emptyset\subset S_1\subset\dots\subset
  S_r\subseteq X$, and let $\lm_C$ be the set of leading monomials of
  the age algebra with this chain support.  The plan is essentially to
  realize $\lm_C$ as the linear basis of some monomial ideal of a
  polynomial ring, so that the generating series of $\lm_C$ is
  realized as an Hilbert series. Consider the polynomial ring
  $\K[S_1,\dots,S_l]$, with its embedding in $\K[X]$ by $S_j\mapsto
  x_{S_j}$. Let $\ideal$ be the subspace spanned by the monomials
  $m:=S_1^{r_1}\cdots S_l^{r_l}$ such that $d_i(m)>|E_i|$ for some $i$.
  It is obviously a monomial ideal. When all monomorphic blocks are
  infinite, $\ideal$ is the trivial ideal $\{0\}$. Consider the subspace
  $\K.\lm_C$ of $\K[S_1,\dots,S_l]$ spanned by the monomials in
  $\lm_C$.  Lemma~\ref{lemma.addlayer} exactly states that
  $\ideal[J]:=\K.\lm_C\oplus \ideal$ is in fact also a monomial ideal of
  $\K[S_1,\dots,S_l]$. Applying
  Proposition~\ref{proposition.HilbertSeriesMonomialIdeal}, the
  Hilbert series of $\ideal$ and $\ideal[J]$ are rational fractions of the form
  \begin{displaymath}
    \frac{P(Z)}{ (1-Z^{|S_1|}) \cdots (1-Z^{|S_l|}) }\ .
  \end{displaymath}
  Hence, the same hold for
  $\hilbert_{\K.\lm_C}=\hilbert_{\ideal[J]}-\hilbert_{\ideal}$.
  Furthermore, whenever $S_j$ contains $i$ with $|E_i|<\infty$, the
  denominator $(1-Z^{|S_l|})$ can be canceled out in
  $\hilbert_{\K.\lm_C}$. The remaining denominator divides $(1-Z)
  \cdots (1-Z^k)$.

  By summing up those Hilbert series $\hilbert_{\K.\lm_C}$ over all
  chains $C$ of subsets of $X$, we get the generating series of all
  the leading monomials, that is the Hilbert series of
  $\agealgebra(R)$. Hence, it is a rational fraction of the form
  \begin{displaymath}
    \frac{P(Z)}{ (1-Z) \cdots (1-Z^k) }\ .
  \end{displaymath}

  Recall that, if $f(z)$ is a rational fraction that is analytic at
  zero and has poles at points $\alpha_1,\dots,\alpha_m$, then there
  exists $m$ polynomials $(\pi_j(x))_{j=1}^m$ such that, for $n$ large
  enough,
  \begin{displaymath}
    f_n = [z^n] f(z) = \sum_{j=1}^m \pi_j(n)\alpha_j^{-n}\,;
  \end{displaymath}
  furthermore, the degree of $\pi_j$ is equal to the order of the pole
  of $f$ at $\alpha_j$ minus one (see
  e.g.~\cite[Theorem~IV.9]{Flajolet_Sedgewick.AnalysisOfAlgorithms}).

  Our fraction has a pole of order at most $k$ a $1$. The other poles
  are at roots of unity and are of order at most $k-1$.  Hence
  $\profile_R(n)$ is eventually a quasi-polynomial. Since its growth
  rate is bounded below by $n^{k-1}$
  (Lemma~\ref{lemma.minimalGrowthRate}) we deduce that the pole at $1$
  is of order $k$, that is $P(1)\ne0$. It follows that
  $\profile_R(n)=an^{k-1} + O(n^{k-2})$, for some $a\in \R^+$, as
  desired. \hfill \qed

\section{Proof of Theorem \ref{hereditary/ideals}}\label{section.lemma}
Let $\Omega_{\mu}$ be the class of finite relational structures with
signature $\mu$. The embeddability relation is a quasi order; once
isomorphic structures are identified, $\Omega_{\mu}$ is a
poset. Initial segments of this poset correspond to hereditary
classes. If $\mathfrak C$ is a hereditary class, members of
$\Omega_{\mu}\setminus \mathfrak C$ are \emph{obstructions} to
$\mathfrak C$. The minimal obstructions (minimal
w.r.t. embeddability) are the \emph{bounds} of $\mathfrak
C$. Clearly $\mathfrak C$ is determined by its bounds. Indeed, if
$\mathfrak B$ is a subset of $\Omega_{\mu}$, set $\uparrow \mathfrak
B:=\{S\in \Omega_{\mu}: B\leq S\; \text{for some} \; B\in \mathfrak
B\}$ and $Forb(\mathfrak B):=\Omega_{\mu}\setminus \uparrow \mathfrak
B$. Then $\mathfrak C=Forb(\mathfrak B)$ where $\mathfrak B$ is the
set of bounds of $\mathfrak C$. Hence, the fact that $\mathfrak C$ can
be defined by a finite number of obstructions amounts to the fact that
it has only finitely many bounds (considered up to isomorphy). An
\emph{ideal} of a poset is a non-empty and up directed initial
segment.  Clearly, the age of a relational structure is an ideal; the
converse holds provided that $\mu$ is finite \cite{Fraisse.1954}. By
extension, the bounds of a relational structure are the bounds of its
age.  The decomposition of a poset, or of an initial segment thereof,
into ideals is the backbone of the theory of ordered sets and the
proof of Theorem \ref{hereditary/ideals} starts with well known
properties of such decompositions.

In the sequel we suppose that $\mu$ is finite, despite that some of the results hold without this requirement. We consider two cases:

Case 1: $\mathfrak C$ contains no infinite antichain
(w.r.t. embedability). Then, $\mathfrak C$ is a finite union of ages.
This is a special case of a general result about posets of
Erd\"os-Tarski~\cite{Erdos_Tarski.1943}. The statement follows.

Case 2: $\mathfrak C$ contains an infinite antichain. Then it contains
an age which cannot be defined by finitely many obstructions and
contains no infinite antichain~\cite[3.9 p.~329]{Pouzet.RPE.1981}
(this fact is a special instance of a property of posets which is
similar to Nash-William's lemma on minimal bad
sequences~\cite{NashWilliams.1963}). With this in mind,
Theorem~\ref{hereditary/ideals} is a consequence of the following
lemma.
\begin{lemma}
  \label{lemma.age.polynomially_bounded}
  \emph{An age with polynomially bounded profile can be defined by
    finitely many obstructions}.
\end{lemma}

Lemma~\ref{lemma.age.polynomially_bounded} is one of the many
properties of relational structures with polynomially bounded profile.
These properties were stated in the thesis of the first author
\cite{Pouzet.TR.1978}; some have been published in
\cite{Pouzet.RPE.1981}. They are presented in the survey
\cite{Pouzet.2006.SurveyProfile} with a complete treatment of the case
of binary structures. For the reader's convenience, the remainder of
this section contains a sketch of the proof of
Lemma~\ref{lemma.age.polynomially_bounded}, in two steps:
\begin{lemma}[{\cite[Theorem~2.12]{Pouzet.2006.SurveyProfile}\footnote{Beware
      that it is stated there that $R$ itself is almost multichainable which
      might actually be an overstatement.}}]
  \label{lemma.age_polynomially_bounded_multichainable}
  If the profile of a relational structure $R$ is bounded by a
  polynomial then there is some $R'$ with the same age which is almost
  multichainable.
\end{lemma}
\begin{lemma}[{\cite[Theorem~4.20]{Pouzet.2006.SurveyProfile}}]
  \label{lemma.almost_multichainable_age}
  The age of an almost multichainable structure can be defined by
  finitely many obstructions.
\end{lemma}

Lemma~\ref{lemma.almost_multichainable_age} relies on the following
notion and result which are are exposed in Fraïssé's book (see
Chapter~13 p.~354, \cite{Fraisse.TR.2000}).
A class $\mathfrak C$ of finite structures is \emph{very beautiful} if
for every integer $k$, the collection $\mathfrak C (k)$ of structures
$(S, U_1, \dots, U_k)$, where $S\in \mathfrak C$ and $U_1, \dots, U_k$
are unary relations with the same domain as $S$, has no infinite
antichain w.r.t. embeddability. A straightforward consequence of
Higman's theorem on words (see
\cite{Higman.1952}) %
is that \emph{the age of an almost multichainable structure is very
  beautiful}. We conclude using that \emph{a very beautiful age can be
  defined by finitely many
  obstructions}~\cite{Pouzet.1972}.%

The proof of Lemma~\ref{lemma.age_polynomially_bounded_multichainable}
uses the notions of kernel and of height.
The \emph{height} of an age $\mathcal A$ is an ordinal, denoted by $h(\mathcal A)$ and defined by induction as follows: $h(\mathcal A)=0$ if $\mathcal A$ has no proper sub-age; that is $\mathcal A$ is the age of the empty relational structure. Otherwise $h(\mathcal A):=\sup\{h(\mathcal A')+1: \mathcal A' \; \text{is a proper sub-age of}\;  \mathcal A\}$. Clearly, the height is defined if and only if there is no strictly descending infinite sequence of sub-ages of $\mathcal A$.  With this definition, it is easy to check that an age $\mathcal A$ has height $n$ with $n\in \N$ if and only if $\mathcal A$ is the age of a relational structure on $n$ elements. Denote by $\omega$   the first infinite ordinal, set  $\omega.2:=\omega+\omega$, $\omega^2:= \omega+\omega+\cdots$. For example, we have $h(\mathcal A)<\omega.2$ if and only if $\mathcal A$ is the age of an almost chainable relational structure \cite{Pouzet_Sobrani.2001}.

Lemma~\ref{lemma.age_polynomially_bounded_multichainable} follows from
the next two facts:

\noindent
\textbf{Fact 1} (\cite[Theorem~4.30]{Pouzet.2006.SurveyProfile}) %
The profile $\profile_R$ grows as a polynomial of degree $k$ if and
only if $\omega.(k+1)\leq h(\mathcal A(R))<\omega.(k+2)$.

\noindent
\textbf{Fact 2} (\cite[Theorem~4.24]{Pouzet.2006.SurveyProfile})%
If $h(\mathcal A)<\omega^2$ then $\mathcal A$ is the age of an almost
multichainable structure.

We need two more facts before sketching the proofs of Facts 1 and 2:

\noindent
\textbf{Fact 3}. If $\mathcal A$ is inexhaustible then either $\mathcal A$ is the age of a multichainable structure and $h(\mathcal A)<\omega^2$ or $\mathcal A$ contains the age $\mathcal A_k$ of a multichainable structure  with $h(\mathcal A_k)\geq \omega.(k+1)$ for every integer $k$. 

\noindent
\textbf{Fact 4}. Let $R$ be a relational structure; if every sub-age of $\mathcal A(R)$  is very beautiful then either $h(A)\geq \omega^2$ or $K(R)$ is finite. 

Fact 3 is  easy (see  Proposition 4.23 of \cite{Pouzet.2006.SurveyProfile}). Fact 4 is deeper. It is in  \cite{Pouzet.TR.1978}. It is not known if the first part of the conclusion of Fact 4 can be dropped. It can if $R$ is made of binary structures (see Theorem 4.24 of \cite{Pouzet.2006.SurveyProfile}). 

The proof of Fact 2 is by induction on $\alpha:=h(\mathcal A)$, with
$\alpha= \omega.(k+1)+p$. From the induction hypothesis, every proper
sub-age is the age of an almost multichainable structure; thus is very
beautiful. By Fact 4, the kernel of any relational structure $R$ with
age $\mathcal A$ is finite. Write $F:=K(R)$ and $E':=E\setminus K(R)$.
Let $M$ be a relational structure with base $E'$ whose local
isomorphisms are those of $R_{\restriction E'}$ which can be extended
by the identity on $F$ to local isomorphisms of $R$. Then, $M$ is
age-inexhaustible and $h(\mathcal A(M))=\beta$ with
$\beta:=\omega.(k+1)$. By Fact 3, $\mathcal A(M)$ is the age of a
multichainable structure. Thus, $\mathcal A(R)$ is the age of an
almost multichainable structure, as desired.

For the proof  of Fact 1,  we suppose first that $R$ is 
 almost multichainable.  As in the proof of Fact 2, we replace $R$ by
 $M$ and we prove that $\varphi_M$ is
 bounded by a polynomial of degree $k$ if and only if $h(\mathcal A
 (M))=\omega.(k+1)$.  For that, we prove first that if  $h(\mathcal
 A(M))=\omega.(k+1)$ then  $\varphi_{M}(n)\leq {n+k\choose k}$  for
 every integer $n$ (Lemma 4.27 of \cite{Pouzet.2006.SurveyProfile}).
 Next,  in a similar way as in the proof of Theorem 2.17, we prove
 that $\varphi_{M}$ is bounded from below by a polynomial of degree
 $k$. Since $\varphi_R$ is bounded by a polynomial of the same degree
 as $\varphi_{M}$, we get the equivalence stated in Fact 1 when $R$ is
 almost multichainable. Now, according to Fact 2, if  $h(\mathcal
 A(R))<\omega.(k+2)$ there is some $R'$ with the same age as $R$
 which is  almost multichainable, and thus $\varphi_R=\varphi_{R'}$ is
 bounded by a polynomial of degree $k$. Conversely, suppose that
 $\varphi_R$ grows as a polynomial of degree $k$. We claim that
 $\mathcal A(R)$ has an height. Indeed, otherwise $\mathcal A(R)$
 contains an infinite antichain. Hence, as we saw in Case 2,   it
 contains an age $\mathcal A'$ with no infinite antichain which cannot
 be defined by finitely many obstructions. This age  has an height,
 say $\alpha'$. If $\alpha'<\omega^2$ then, by Fact 2, $\mathcal A'$
 is the age of  an almost multichainable structure, and thus can be defined by finitely many obstructions, which is impossible.  Thus $\alpha \geq\omega^2$. But then by Fact 3, $\mathcal A'$  contains ages of height $\omega. (k'+1)$ for every $k'$ and thus $\varphi_R$ grows faster than every polynomial. A contradiction. This proves our claim. Now, let $\alpha=h(\mathcal A(R))$. By the same token, $\alpha<\omega^2$, hence $\mathcal A(R)$ is the age of  an almost multichainable structure. And thus $\omega.(k+1)\leq \alpha<\omega.(k+2)$, as desired.

\appendix

\section{Examples of relational structures and age algebras}

\label{appendix.examples}

\subsection{Examples coming from graphs and digraphs}

A graph $G:=(V, \mathcal E)$ being considered as a binary irreflexive
and symmetric relation, its profile $\profile_G$ is the function which
counts, for each integer $n$, the number $\profile_G(n)$ of induced
subgraphs on $n$ elements subsets of $V(G)$, isomorphic subgraphs
counting for one.  Graphs with profile bounded by a polynomial have
been described  in~\cite{Pouzet.2006.SurveyProfile,Balogh_Bollobas_Saks_Sos.2009}. According to
\cite{Balogh_Bollobas_Saks_Sos.2009} the profile is either a quasi-polynomial or is
eventually bounded below by the partition function $\wp$, where
$\wp(n)$ is the number of integer partitions of $n$. Tournaments with polynomially  bounded profile have been characterized in~\cite{Boudabbous.Pouzet.2009} as lexicographical sums of acyclic tournaments indexed by a finite tournament. They admit finite monomorphic decomposition, hence their profile is a quasi-polynomial. 

\begin{example}
  \label{example.coclique}
  Trivially, if $G$ is an infinite clique $K_\infty$ or coclique
  $\overline K_\infty$, then there is a single isomorphic type for
  each $d$, hence  $\profile_G(n)=1$. There is one monomorphic component $E_x=E$, the
  age algebra is $\K[x]$, and its Hilbert series is $\frac{1}{1-Z}$.
  If instead one consider digraphs, we recover the same age and
  age algebra from any infinite chain $\N, \Z, \Q, \R$, etc. or
  antichain.
\end{example}

A bit less trivial is the fact that $\profile_G$ is bounded if and
only if $G$ is \emph{almost constant} in the sense of
Fraïssé~\cite{Fraisse.TR.2000} (there exists a finite subset $F_G$
of vertices such that two pairs of vertices having the same
intersection on $F_G$ are both edges or both non-edges).

\begin{example}
  \label{example.symmetricPolynomials}
  Let $G$ be the direct sum $K_\omega\oplus \dots \oplus K_\omega$ of
  $k$ infinite cliques (or chains) $E_1,\dots,E_k$ (see figure in
  Table~\ref{table.overview}). The $E_i$ form the monomorphic components.
  The profile counts the number $\profile_G(n) = p_k(n) \simeq
  \frac{n^k} {(k+1)!k!}$ of integer partitions with at most $k$ parts.
  The age algebra is the ring of symmetric polynomials $\sym(X)$ on
  $k$ variables whose Hilbert series is
  $\frac{1}{(1-Z)\cdots(1-Z^k)}$.%
\end{example}
\begin{examples}
  \label{example.CliquePlusIndependent}
  Let $G$ be the direct sum $K_\omega\oplus \overline K_\omega$ of an
  infinite clique and an infinite independent set. Then, $\profile_G(n)=n$ for
  $n\geq 1$, and
  $\hilbert_G=1+\frac {Z}{(1-Z)^2}=
  \frac{1-Z+Z^2}{(1-Z)^2}=\frac{1+Z^3}{(1-Z)(1-Z^2)}$.
  Hence, the Hilbert series has one representation as a rational
  fraction with a numerator with some negative coefficient, and
  another with all coefficients non-negative.

  This Hilbert series coincides further with that of
  Examples~\ref{example.notfinitelygenerated}, and~\ref{example.qsym}
  for $k=2$. Still, in the first and third case, there are two
  infinite monomorphic components whereas in the second there are three: one finite
  and two infinite. Furthermore, the age algebra is finitely generated
  and even Cohen-Macaulay in the first (take the free subalgebra
  generated by a point and a $2$-chain, and take as module generators
  the empty set and a $3$-chain) and third case, but not in the
  second.
\end{examples}

\begin{example}
  \label{example.notfinitelygenerated}%
  Let $G$ be the direct sum $K_{(1, \omega)}\oplus \overline K_\omega$
  of an infinite wheel and an infinite independent set. There are two
  infinite monomorphic components, $E_1$ the set of leaves of the wheel
  and $E_2$ the independent set, and one finite, $E_3$, containing the
  center $c$ of the wheel. Each isomorphism type consists of a wheel
  and an independent set, so the Hilbert series is
  $\hilbert_G(Z)=(1+\frac{Z^2}{1-Z})\frac{1}{1-Z}=\frac{1-Z+Z^2}{(1-Z)^2}=\frac{1+Z^3}{(1-Z)(1-Z^2)}$.

  What makes this relational structure special is that the monomorphic
  decomposition $(E_1,E_2,E_3)$ is minimal, whereas $(E_1,E_2)$ is
  \emph{not} a minimal monomorphic decomposition of the restriction of
  $R$ to $E_1\cup E_2$. We now prove that this causes the age algebra
  not to be finitely generated.
  Consider the subalgebra $\mathcal B:=\K[e_1(E)]$. In each degree
  $d$, it is spanned by the sum $b_d$ of all subsets of size $d$ of
  $E$. Key fact: any element $s$ of $\agealgebra(R)$ can be uniquely
  written as $s =: a(s) + b(s)$ where $b(s)$ is in $\mathcal B$, and
  all subsets in the support of $a(s)$ contain the unique element $c$
  of $E_3$. Note in particular that $a(s)^2=0$ for any $s$ homogeneous
  of positive degree.
  Let $S$ be a finite generating set of the age algebra made of
  homogeneous elements of positive degree. By the remark above,
  $\{a(s), s\in S\}$ generates $\agealgebra(R)$ as a $\mathcal B$-module.
  It follows that the graded dimension of $\agealgebra(R)$ is bounded by
  $|S|$, a contradiction.
\end{example}

\begin{example}
  \label{example.symmetricFunctions}
  When extending Example~\ref{example.symmetricPolynomials} to infinitely many
  cliques the age algebra becomes the ring $\sym$ of symmetric
  functions whose Hilbert series is
    $\hilbert_R(Z) = \prod_{d\geq 1} \frac{1}{(1-Z^d)}$.%
\end{example}
\begin{example}
  \label{example.path}
  If $G$ is an infinite path, then the finite restrictions are direct
  sums of paths. Therefore, the profile counts the number of integer
  partitions of $n$. The age algebra is the free commutative algebra
  generated by the paths of length $1,2,\dots$, which is again
  isomorphic to $\sym$.  However, this time, the monomorphic
  components are reduced to singletons.
\end{example}

\begin{example}
  \label{example.rado}
  If $G$ is the Rado graph, then $\profile_G(n)$ counts the total
  number of unlabelled graphs. The age algebra is the free commutative
  algebra generated by the connected graphs. Its Hilbert series
  is %
  $\hilbert_R(Z) = \prod_{d\geq 1} \frac{1}{(1-Z^d)^{c_d}}$, where
  $c_d$ is the number of connected graphs with $d$ vertices.
\end{example}

\begin{example}
  \label{example.tournamentC3WreathN}
  Let $G$ be the lexicographic sum tournament obtained by substituting
  each point $i$ of the cycle $C_3:= \{(1,2), (2,3), (3,1)\}$ on
  $\{1,2,3\}$ by the chain $\N$. The three chains $(E_1,E_2,E_3)$ give
  the minimal monomorphic decomposition of $G$, but this decomposition
  is not recursively minimal and the age algebra is not finitely
  generated because, as above, $(E_1,E_2)$ is not the minimal
  monomorphic decomposition of $R$ restricted to $E_1\cup E_2$
  (see~\cite{Pouzet_Thiery.AgeAlgebra2}).
\end{example}

\begin{example}
  \label{example.QwreathG}
  Take the direct sum $\overline K_\infty \wr G$ of infinitely many
  copies of a finite connected graph $G$. The age algebra is the free
  commutative algebra generated by the (finitely many) connected
  induced subgraphs of $G$.
  Taking for $G$ the graph $K_{1,1}$, one gets the infinite
  matching. The age algebra is finitely generated, whereas there are
  infinitely many monomorphic blocks.
  The extension of this example to $G$ a finite relational structure
  is straightforward.
\end{example}

\begin{examples}
  \label{example.notFinitelyGeneratedInfiniteDecomposition}
  Let $G$ be the simple graph consisting of the direct sum of an
  infinite wheel and an infinite matching. Each isomorphism type is
  the direct sum of a wheel, an independent set, and a
  matching. Therefore, the Hilbert series is
  $\hilbert_G(Z)=\frac{Z}{1-Z}\frac{1}{1-Z}\frac{1}{1-Z^2}$, and the
  profile has polynomial growth: $\profile_G(n) \sim a n^2$. There
  is one infinite monomorphic block (the leaves of the wheel), and
  infinitely many finite ones (the center of the wheel, and the edges
  of the matching). The age algebra is not finitely generated because
  $G$ contains as restriction the graph of
  Example~\ref{example.notfinitelygenerated} whose age algebra is not
  finitely generated.
\end{examples}

\subsection{Examples coming from groups} \label {subsection:groups}

We first look at orbital profiles. The fact that they are special  cases of profiles is easy to prove. In fact, 
for every permutation group $G$ on a set $E$, there is a relational
structure $R$ on $E$ such that $\aut
R = \overline{G}$ (the topological closure of $G$ in the symmetric
group $\mathfrak{G}(E)$, equipped with the topology induced by the
product topology on $E^E$, $E$ being equipped with the discrete
topology). In particular, $\theta_G(n)=\profile_R(n)$ for all $n$. %
Oligomorphic groups are quite common.  Indeed, \emph{if $E$ is
  denumerable, then $G$ is oligomorphic if and only if the complete
  theory of $R$ is
  $\aleph_0$-categorical}~\cite{Ryll-Nardzewski.1959}.
Cameron conjectured that the orbital profile $\theta_G$ is polynomial
(in the sense that $\theta_G(n)\sim a n^k$) provided that it is
bounded by some polynomial. This particular consequence of
Conjecture~\ref{conjecture.quasipolynomial} has not been solved yet.

\begin{example}
  Let $G$ be the trivial group on an $m$ element set $E$. Set $R:=
  (E, u_1,\dots,u_{m})$, where each $u_i$ is a unary
  relation defining the $i$-th element of $E$. Then, $\theta_G(n)=
  \profile_R (n)= {m \choose n}$.
\end{example}

\begin{example}
  Let $G$ be the symmetric group $\sg_N$, acting on the set of all
  pairs $\{i,j\}$ of $\{1,\dots,N\}$ by $\sigma(\{i,j\}) =
  \{\sigma(i),\sigma(j)\}$. Then, the orbits of $G$ are the unlabelled
  graphs on $N$ vertices, counted by number of edges. In the age
  algebra, the product of two graphs is the sum (with multiplicities)
  of all graphs that can be obtained by superposing them without
  overlapping edges.
\end{example}

\begin{example}
  Let $\age$ be a finite alphabet with $k$ elements, and let
  $\age^*$ be the set of words over $\age$. Then each word
  can be viewed as a finite chain coloured by $k$ colors. Hence
  $\age^*$ is the age of the relational structure $R$ made of
  the chain $\Q$ of rational numbers divided into $k$ colors in such a
  way that, between two distinct rational numbers, all colors
  appear. Furthermore, $R$ is homogeneous in the sense that every
  local isomorphism of $R$ with finite domain extends to an
  automorphism of $R$, hence the set of orbits of $G: =\aut(R)$ can be
  identified to $\age(R)$. As pointed out by Cameron
  \cite{cameron.1997}, the age algebra $\agealgebra(R)$ is isomorphic
  to the shuffle algebra over $\age$, an important object in
  algebraic combinatorics (see \cite{Lothaire.1997}). A more
  sophisticated example of shuffle algebra is presented in
  Subsection~\ref{gerritzen}.
\end{example}

\begin{example}
  Let $G := \aut\,\Q$, where $\Q=(\Q,\leq)$ is the chain of rational
  numbers. Then, $\theta_G(n) = \profile_{\Q}(n)= 1$ for all $n$.
  There is a single monomorphic block, and $\agealgebra(R)\approx\K[x]$.
\end{example}
\begin{example}
  \label{example.polynomials}
  Let $R:=(\Q, \leq, u_1,\dots,u_k)$, where $\Q$ is the chain of
  rational numbers, and $u_1,\dots,u_k$ are $k$ unary relations which
  divide $Q$ into $k$ non-trivial intervals $E_1,\dots,E_k$. Then, $\profile_R(n)
  = {n+k-1 \choose k-1}$ and $\hilbert_R=\frac{1}{(1-Z)^k}$.  The
  $E_i$'s are the monomorphic blocks and $\agealgebra(R)\approx\K[X]$.
\end{example}

\begin{example}
  \label{example.permgroup}
  Let $G'$ be the wreath product $G':=G\wr \sg_\N$ of a permutation
  group $G$ acting on $\{1,\ldots, k\}$ and of $\sg_\N$, the symmetric
  group on $\N$. Looking at $G'$ as a permutation group acting on
  $E':=\{1,\ldots, k\}\times \N $, then $G'=\aut R'$ for some
  relational structure $R'$ on $E'$; moreover, for all $n$,
  $\theta_{G'}(n)=\profile_{R'}(n)$. Among the possible $R'$ take
  $R\wr \N:=(E', \equiv, (\overline \rho_i)_{i\in I})$, where $\equiv$
  is $\{((i, n),(j,m))\in E'^{2}\suchthat i=j\} $, $\overline \rho_i:=\{
  ((x_1, m_1),\dots,(x_{n_i}, m_{n_i}))\suchthat (x_1,\dots, x_{n_i})\in
  \rho_i, ( m_1, \dots, m_{n_i})\in \N^{n_i} \}$, and $R:=(\{1,\dots,
  k\}, (\rho_i)_{i\in I})$ is a relational structure having signature
  $\mu:= (n_i)_{ i\in I}$ such that $\aut R= G$.  The relational
  structure $R\wr \N$ decomposes into $k$ monomorphic blocks,
  namely the equivalence classes of $\equiv$.
  
  As it turns out~\cite{Cameron.1990}, $\hilbert_{R\wr \N}$ is the
  Hilbert series %
  of the \emph{invariant ring} $\K[X]^G$ of $G$, that is the subring
  of the polynomials in $X$ which are invariant under the action of
  $G$. In fact, the identification of the age algebra as a subring of
  $\K[X]$ gives an isomorphism with $\K[X]^G$. As it is well known,
  this ring is Cohen-Macaulay, and the Hilbert series is a rational
  fraction of the form given in Theorem~\ref{theorem.quasipolynomial},
  where the coefficients of $P(Z)$ are non-negative.

  When $G$ is the trivial group, one recovers the polynomial ring
  $\K[X]$, as in Example~\ref{example.polynomials}.
\end{example}

\begin{problem}
  Find an example of a permutation group $G'$
  with no finite orbit, such that the orbital profile of $G'$ has
  polynomial growth, but the generating series is not the Hilbert
  series of the invariant ring $\K[X]^G$ of some permutation group $G$
  acting on a finite set $X$.
\end{problem}

\subsection{Examples coming from permutation groupoids}

Let $X$ be a set. A \emph{local bijection} of $X$ is a bijective
function $f$ whose domain $\dom f$ and image $\im f$ are subsets of
$X$. %
The inverse $f^{-1}$ of a local bijection $f$, its restriction
$f_{\restriction X'}$ to a subset $X'$ of $\dom f$ (with codomain
$f(X')$), and the composition $f\circ g$ of two local bijections $f$
and $g$ such that $\im g=\dom f$ are defined in the natural way.  A
set $G$ of local bijections of $X$ is called a \emph{permutation
  groupoid} if it contains the identity and is stable by restriction,
inverse, and composition. It is furthermore \emph{locally closed} if a local
bijection $f$ is in $G$ whenever all its finite restrictions
are. Obviously, the closure $\isection G$ of a permutation group $G$
by restriction is a permutation groupoid. More interestingly, the
local isomorphisms of a relational structure form a locally closed
permutation groupoid, and reciprocally, \emph{any locally closed
  permutation groupoid $G$ can be obtained from a suitable relational
  structure $R_G$ on $X$}.

The wreath product construction of an age algebra matching the
invariant ring $\K[X]^G$ of a permutation group $G$ (see
Example~\ref{example.permgroup}) can be extended straightforwardly to
\emph{permutation groupoids}. Many, but not all, properties of
invariant rings of permutation groups carry over (see
Table~\ref{table.overview}
and~\cite{Pouzet_Thiery.AgeAlgebra.2005,Pouzet_Thiery.AgeAlgebra2});
in particular, the invariant ring is still a module over symmetric
functions, but not necessarily Cohen-Macaulay.

\begin{examples}
  \label{example.qsym}
  Take $n\in\N\cup \{\infty\}$ and let $G$ be the permutation groupoid
  of the strictly increasing local bijections of $\{1,\dots, n\}$, or
  equivalently of the local isomorphisms of the chain
  $1<\cdots<n$. Then, $\K[X]^G$ is the ring $\qsym(X)$ of
  \emph{quasi-symmetric polynomials} on the ordered alphabet $X$, as
  introduced by I.~Gessel~\cite{Gessel.QSym.1984}. As shown by
  F.~Bergeron and C.~Reutenauer, $\hilbert_{\qsym(X)}=\frac
  {P_n(Z)}{(1-Z)(1-Z^2)\cdots(1-Z^n)}$, where the coefficients of
  $P_n(Z)$ are non negative. In fact, the ring is
  Cohen-Macalay~\cite{Garsia_Wallach.2003}.

  Taking the same groupoid $G$, and letting it act naturally on
  respectively pairs, couples, $k$-subsets, or $k$-tuples of elements
  of $\{1,\dots,n\}$, yield respectively the (un)oriented (hyper)graph
  quasi-symmetric polynomials of~\cite{Novelli_Thibon_Thiery.2004}.

  The $r$-quasi symmetric polynomials~\cite{Hivert.RQSym.2004} can be
  realized as well as the age algebra of a relational structure (but
  not as the invariant ring of a permutation groupoid). Namely start
  from the relational structure of
  Example~\ref{example.symmetricPolynomials}, and add another $2r$-ary
  relation $\rho$ such that $\rho(x_1,\dots,x_r,y_1,\dots,y_r)$ holds
  if $x_1,\dots,x_r$ are distinct and in some block $E_i$ and
  $y_1,\dots,y_r$ are distinct and in some block $E_j$ with $i<j$. For
  $r=1$, one recovers the relational structure giving quasi symmetric
  functions and for $r=0$ 
  the relational structure giving symmetric polynomials.
\end{examples}

\begin{example}
  \label{example.nonCM.groupoid}
  Let $G$ be the permutation groupoid on $\{1,2,3\}$ generated by the
  local bijection $1\mapsto 2$. Then, $G$ is the restriction of the
  finite permutation group $\langle (1,2), (3,4)\rangle$ whose
  invariant ring is Cohen-Macaulay. However, the age algebra $\K[X]^G$
  itself is not Cohen-Macaulay.
  In fact, the numerator of the Hilbert series cannot be chosen with
  non-negative coefficients. Indeed, $\hilbert_{\K[X]^G} =
  \frac{1-Z+2Z^2-Z^3}{(1-Z)^3}$, and the coefficient of highest degree
  in the product of the numerator by
  $\frac{(1-Z^{n_1})(1-Z^{n_2})(1-Z^{n_3})}{(1-Z)^3}$ is always $-1$.
\end{example}

\subsection{Example: the shuffle algebra of planar tree polynomials}
\label{gerritzen}

As a final example, and in order to illustrate the limits of
monomorphic decompositions, we consider the \emph{shuffle algebra of
  planar tree polynomials} $(\K\{x\}_\infty, \shuffle)$ which arises
naturally in the study of non associative analogues of the exponential
and
logarithm~\cite{Drensky_Gerritzen.2004,Gerritzen.2005.PlanarShuffleProduct}. We
realize $(\K\{x\}_\infty, \shuffle)$ as an age algebra, and show that
the minimal monomorphic decomposition of the underlying relational
structure is trivial, and in particular infinite.

In this section, all trees are rooted, ordered, and unlabelled (in the
papers cited above, those trees are called \emph{planar}). A tree is
\emph{reduced} if all its internal nodes are of arity at least
two. Let $\PRT_i$ be the set of all reduced trees with $i$ leaves and set
$\PRT:=\bigcup_{d=0}^\infty \PRT_d$.
By convention, $\PRT_0$ contains the \emph{empty reduced tree} with
  zero leaves.  Denoting leaves by ``$\leaf$'' and subtrees using
parentheses, one has:
\begin{displaymath}
  \PRT_1=\{\ \leaf\ \}, \qquad
  \PRT_2=\{\ (\leaf,\leaf)\ \}, \qquad
  \PRT_3=\{\ (\leaf,(\leaf,\leaf)), \  (\leaf,\leaf,\leaf), \ ((\leaf,\leaf),\leaf)\ \}\,.
\end{displaymath}
Those trees are counted by the sequence of Schröder numbers or super
Catalan numbers (\#A001003 of~\cite{OEIS}):
\begin{displaymath}
  1, 1, 1, 3, 11, 45, 197, 903, 4279, 20793, 103049, ...
\end{displaymath}
To each tree $\tree$ can be associated a canonical reduced tree
$\reduced(\tree)$ by contracting all paths in $\tree$ to suppress
intermediate nodes of arity $1$.
Given a subset $A$ of the leaves of $\tree$, one defines the
\emph{contraction of $\tree$ on $A$ as
  $\tree_{\restriction A}:=\reduced(\tree')$}, where $\tree'$ is the subtree
induced by $\tree$ on the set of all nodes of $\tree$ between the root
and the leaves in $A$.
\begin{lemma}
  \label{lemma.3reconstruction.reduced_tree}
  Let $\tree$ and $\tree'$ be two reduced trees with $d$ leaves. Then,
  $\tree=\tree'$ if and only if the contractions $\tree_{\restriction A}$ and
  $\tree'_{\restriction A}$ are identical for any $3$-subset $A$ of the leaves.
\end{lemma}
\begin{proof}
  The ``only if'' statement is obvious, and we turn to the ``if''
  statement. For simplicity we denote the leaves
  $\{1,\dots,d\}$. Consider an internal node of $\tree$; since $\tree$
  is reduced, this node is uniquely caracterized by the interval
  $[i,j]$ formed by the leaves under it. We call $[i,j]$ a \emph{node
    interval}. Note that a reduced tree is uniquely caracterized by
  the collection of its node invervals. Note further that an interval
  $[i,j]$ is a node interval if and only if:
  \begin{itemize}
  \item for $k<i$, the restriction $\tree_{\restriction\{k,i,j\}}$ is the tree $(\leaf,(\leaf,\leaf))$, and
  \item for $k>j$, the restriction $\tree_{\restriction\{i,j,k\}}$ is the tree $((\leaf,\leaf),\leaf)$.
  \end{itemize}
  This concludes the proof.
\end{proof}

We now construct a relational structure whose isomorphism types are
given by the reduced trees. Consider the infinite rooted ordered tree
$T$ such that each internal node has infinitely many children, each
alternatively a leaf or a copy of $T$ (see
Figure~\ref{figure.infinite_planar_tree}). Let $E$ be the set of
leaves of $T$. To each finite subset $A$ of $E$, we associate the
reduced tree $T_{\restriction A}$ obtained by contracting $T$ on $A$
as in the finite case (see the example in
Figure~\ref{figure.infinite_planar_tree}).
\begin{figure}[h]
  \centering
  \includegraphics[width=\textwidth]{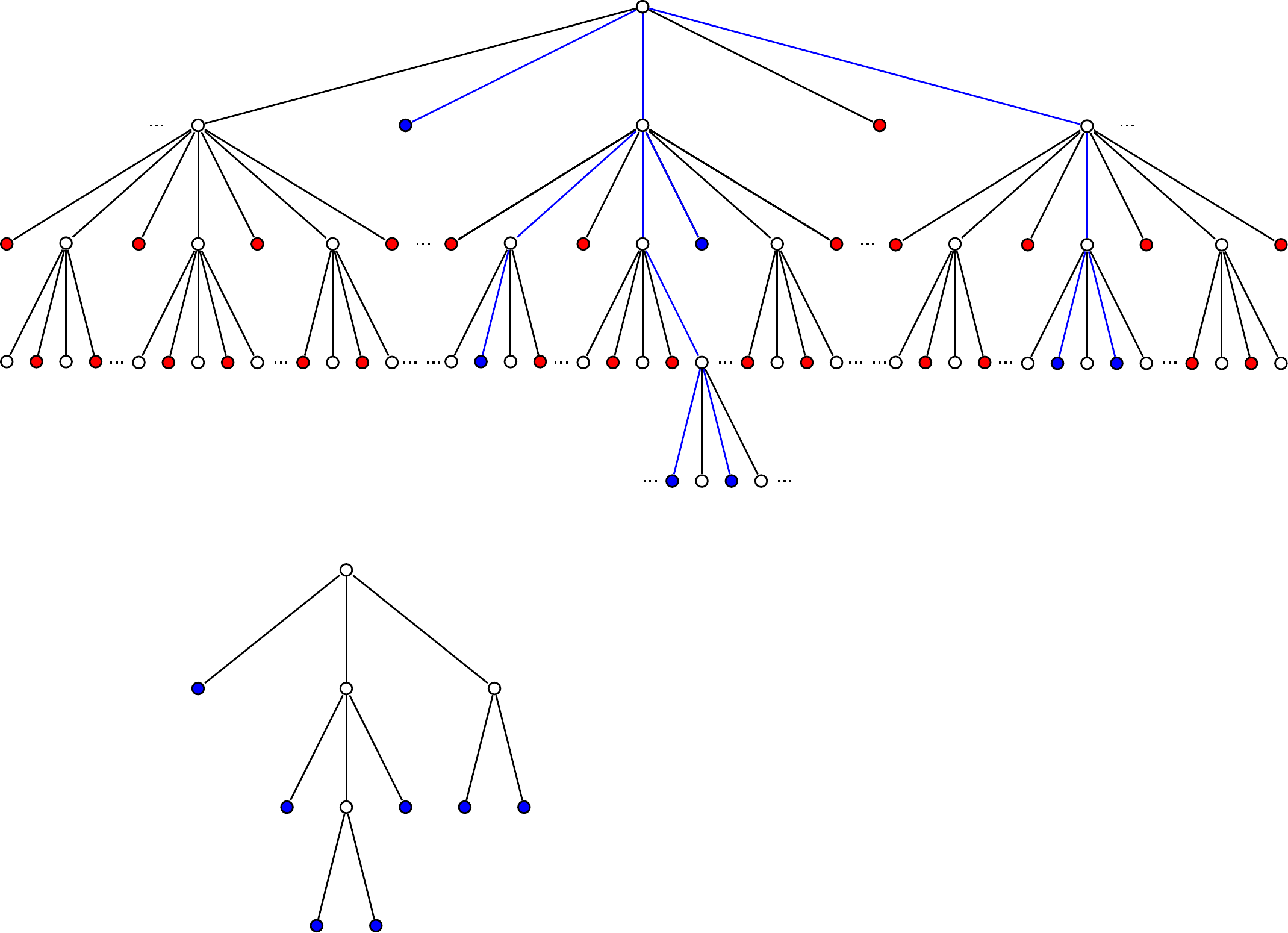}
  \caption{Above: the infinite tree $T$; in white, its internal nodes;
    in red and blue: its leaves. Below: the reduced tree
    $T_{\restriction A}$,
    where $A$ is the set of blue leaves in $T$.}
  \label{figure.infinite_planar_tree}
\end{figure}

Let $R$ be the relational structure obtained by endowing $E$ with:
\begin{itemize}
\item the total order $<$ induced by left-right infix order on $T$;
\item for each of the three reduced trees $\tree$ in $\PRT_3$, the
  ternary relation
  \begin{displaymath}
    \rho_\tree := \{ (x,y,z) \suchthat x<y<z \text{ and } T_{\restriction\{x,y,z\}} = \tree \}\,.
  \end{displaymath}
\end{itemize}
Note that the relational structure $(E,<)$ is isomorphic to the chain
of rationals.

\begin{proposition}
  The profile of $R$ counts the reduced trees. Namely, $A\approx A'$
  if and only if $T_{\restriction A}=T_{\restriction A'}$, and any
  reduced tree $\tree\in \PRT$ arises this way.
\end{proposition}
\begin{proof}
  We need only to consider the case where $A$ and $A'$ are of the same
  size $d$. Write $A=\{x_1,\dots,x_d\}$ and $A'=\{x'_1,\dots,x'_d\}$
  along the total order $<$. Take $\{i,j,k\}$ a $3$-subset of leaves
  of $T_{\restriction A}$ (and of $T_{\restriction A'}$). Thanks to the
  compatibility of contraction with $<$, one has
  $(T_{\restriction A})_{\restriction\{i,j,k\}}=T_{\restriction\{x_i,x_j,x_k\}}$
  and similarly for $A'$.  Therefore $A$ and $A'$ are isomorphic if
  and only if $T_{\restriction A}$ and $T_{\restriction A'}$ have the
  same $3$-leaf contractions. We conclude by reconstruction using
  Lemma~\ref{lemma.3reconstruction.reduced_tree}.

  For the last statement, choose any of the straightforward embeddings
  of $\tree$ in $T$.
\end{proof}

\begin{proposition}
  The minimal monomorphic decomposition of $R$ is trivial: each of its
  monomorphic component is a singleton.
\end{proposition}
\begin{proof}
  Since a
  subset of a monomorphic part is a monomorphic part (see
  Lemma~\ref{lemma.monomorphic_parts_are_good}), it is sufficient to
  prove that there is no two-element monomorphic part. Take $a<b$ in
  $E$. In between $a$ and $b$ in $T$ there is a full-blown copy of
  $T$.  Thus, we can take two leaves $c,d$ of $T$ with $a<c<d<b$ such
  that $c,d$ have a common ancestor which is neither an ancestor of
  $a$ nor of $b$. Then,
  \begin{displaymath}
    T_{\restriction\{a,c,d\}}\ =\ (\leaf,(\leaf,\leaf))\ \ne\ 
    ((\leaf,\leaf),\leaf)\ =\ T_{\restriction\{b,c,d\}}\,.
  \end{displaymath}
  Therefore $\{a,b\}$ is not a monomorphic part.
\end{proof}
\begin{proposition}
  \label{proposition.planar_shuffle_algebra}
  The age algebra of $R$ is isomorphic to the shuffle algebra of
  planar tree polynomials $(\K\{x\}_\infty, \shuffle)$.
\end{proposition}
\begin{proof}
  Let $\tree, \tree_1,\tree_2$ be three reduced trees. The structure
  coefficient $c_{\tree_1,\tree_2}^{\tree}$ is obtained by taking any
  $A$ such that $\tree=T_{\restriction A}$ and counting the number of
  $A_1\uplus A_2=A$ such that $\tree_1=T_{\restriction A_1}$ and
  $\tree_2=T_{\restriction A_2}$. This matches with the definition of the
  structure constants of the shuffle product on planar tree
  polynomials
  (see~\cite[Section~3]{Gerritzen.2005.PlanarShuffleProduct}).
\end{proof}

\bibliographystyle{alpha}
\bibliography{main1}

\end{document}